\documentclass{article}

\usepackage{amsmath,amssymb,amsthm,graphicx,color,sidecap,wrapfig}
\usepackage[inner=32mm,outer=32mm,top=28mm,bottom=28mm]{geometry}
\usepackage{framed,color,perpage}

\newcommand{\ve}{\varepsilon}

\def \div {{\rm div}}

\MakePerPage{footnote}

\newtheorem{thm}{Theorem}[section]
\newtheorem{rem}[thm]{Remark}
\newtheorem{prop}[thm]{Proposition}
\newtheorem{lem}[thm]{Lemma}
\newtheorem{cor}[thm]{Corollary}
\newtheorem{conj}[thm]{Conjecture}

\title{Stability of steady states and bifurcation to
	traveling waves in a free boundary model of cell motility}

\author{Leonid Berlyand\thanks{Department of Mathematics, Huck Institutes of Life   Sciences and Materials Research Institute at the  Penn State University,
		University Park, PA, 16802, USA.} \and Volodymyr Rybalko\thanks{B.Verkin Institute for Low Temperature Physics and Engineering of NASU, 
		47 Nauky ave, Khariv 61103, Ukraine.}}

\begin{document}

\maketitle


\begin{abstract}
{We introduce a two-dimensional coupled Hele-Shaw/ Keller-Segel type free boundary model for motility of eukaryotic cells on substrates.
	The key ingredients of this model are the Darcy law for overdamped motion of the cytoskeleton (active) gel and Hele-Shaw type boundary conditions (Young-Laplace equation for pressure and continuity of velocities).
		We first  show that  radially symmetric steady state solutions become unstable and bifurcate  to traveling wave  solutions. 
		Next we  establish linear and nonlinear stability of the steady states. 
		  We show that  linear stability analysis is inconclusive for both steady states and traveling waves. Therefore we  use invariance properties to prove nonlinear stability of steady states.}
\end{abstract}

\section{Introduction}



Motion  of living cells has been the subject of extensive  studies in biology, soft-matter physics and more recently in mathematics.   Living cells     are  primarily driven by cytoskeleton  gel dynamics. The study of  cytoskeleton  gels  led to a recent development of the so-called  
``Active gel physics", see \cite{ProJulJoa2015}.

The key element  of this motion is cell polarity, which enables cells to carry out specialized functions and therefore is a fundamental issue in cell biology.  Also motion of specific cells such as keratocytes  in the cornea   is of medical relevance as they are involved, e.g., in wound healing after eye surgery or injuries.    Moreover keratocytes are perfect for experiments and modeling since they are naturally found on 
flat surfaces, which allows capturing 
the main features of their motion by spatially two dimensional models. The typical
modes of motion of keratocytes in both  cornea and fishscales are rest (no movement at all) or steady motion with
fixed shape, speed, and direction \cite{Ker_etal2008,
BarLeeAllTheMog2015}. That is why it is important  to study the
steady states and traveling waves that describe resting cells  and 
steadily moving cells respectively.

The two leading  mechanisms of cell motion are  protrusion generated by polymerization of actin filaments (more precisely, filamentous 
actin or F-actin) and  contraction  due to myosin motors \cite{Ker_etal2008}.
  Our goal is to study  the contraction driven cell motion  when   polymerization is negligible since it is balanced by depolymerization (complementary work on polymerization without myosin contraction, see  \cite{EtcMeuVoi2017}, \cite{MeuPrivat2018}). To this end we  introduce and  investigate  a 2D model with free boundary that generalizes 1D  Keller-Segel type free boundary model from 
 \cite{RecPutTru2013}, \cite{RecPutTru2015}. 
 Even though the boundary in 1D is simply two points, our analysis shows that several key qualitative  properties established in  \cite{RecPutTru2013}, \cite{RecPutTru2015}  are also observed in 2D.  However, the transition from 1D  to 2D requires addressing new issues  such as modeling and analysis of evolution of the  domain shape. For instance, the  problem on a moving interval of variable length (1D domain with free boundary) is reduced to a problem on a fixed interval via a linear change of variable, whereas in 2D case 
such a reduction requires a much more sophisticated nonlinear change of variables.  

Two-dimensional active gel models with free boundary were introduced in, e.g., \cite{BarLeeAllTheMog2015},  \cite{CalJonJoanPro2008}, \cite{BlaCas2013}.   
The problems in  \cite{CalJonJoanPro2008}, \cite{BlaCas2013}
model the polymerization  driven cell motion   when   myosin contraction is dominated by polymerization, which naturally  complements present work.  Although the model from \cite{BlaCas2013}  looks similar to the classical Hele-Shaw model, the two are  different in some fundamental aspects such as  presence of persistent motion modeled by traveling wave solution. More recently  a 2D model of the intracellular dynamics with fixed cell shape as a disc was introduced and analyzed numerically analytically in \cite{EtcMeuVoi2019}. 

A free boundary 2D model  introduced  and analyzed numerically in   \cite{BarLeeAllTheMog2015} accounts for both polymerization  and myosin contraction. This model was studied analytically in \cite{BerFuhRyb2018} where the traveling wave solutions were established.  It was also shown in \cite{BerFuhRyb2018} that this model reduces   to  the Keller-Segel  system in a free boundary setting. This system  in fixed domains  appears in various chemotaxis models and it has been extensively studied in mathematical literature due to the finite time blow-up phenomenon caused by the cross-diffusion term (\cite{TaoWin2012},p.1903)  in dimensions 2 and higher. The fact that the minimal model of contraction  driven motility reduces to a Keller-Segel system with free boundaries was first realized  in \cite{RecPutTru2013} where the corresponding traveling wave solutions were analyzed in the simplest 1D setting.

 While in the model  \cite{BarLeeAllTheMog2015}
 the kinematic condition at the free boundary contains curvature, in present work we assume continuity of velocities of the gel and the membrane (boundary)  as in 1D  model  \cite{RecPutTru2013}, \cite{RecPutTru2015})  but adapt the Young-Laplace equation for the pressure on the boundary as usually done in Hele-Shaw model and has no analog in 1D.

Our objective is analysis of the coupled Hele-Shaw/Keller-Segel model. Specifically, we are interested in existence and stability of its special solutions such as steady states and traveling waves, which are important for understanding cell motility. While the existence of radially symmetric steady  states is straightforward, their nonlinear stability analysis is highly non-trivial. Indeed, we first perform the linear stability analysis around radial steady states and show that the linearized operator has zero eigenvalue of multiplicity 2. 
The corresponding two eigenvectors appear since these steady states are a continuum family parametrized by their centers (shift invariance) and radii. Thus, the linear stability analysis is inconclusive. For nonlinear stability we need to control the component of the solution corresponding to the both eigenvectors.  For the first eigenvector we use  factorization in shifts for the linearized problem, whereas for the second one we use conservation of total myosin mass in place a Lyapunov function, which is a standard tool in proof of nonlinear stability (it is known that  
establishing Lyapunov function in free boundary problems is quite difficult). Another challenge in the proof of nonlinear stability of steady states can be  described as follows. The problem with free boundary is reduced to a problem in a fixed domain (a disk). For classical Hele-Shaw problem, this is done by conformal maps since the pressure is harmonic and therefore the PDE is conformally invariant 
\cite{ConPug1993}, \cite{BlaCas2013}. However, the pressure in our problem, see \eqref{actflow_in_terms_of_phi}-\eqref{myosin_bc_I}, is not harmonic due to coupling with myosin density. Similar difficulty arises  in tumor growth free boundary problems, see, e.g.,  \cite{FrHu2006}, \cite{BazFri2003}, where it is dealt with by applying the Hanzawa transform. However, the Hanzawa 
transform can not be used in problem 
\eqref{actflow_in_terms_of_phi}--\eqref{myosin_bc_I} due to the Neumann condition  \eqref{myosin_bc_I}.
Indeed, this transform does not preserve normal derivative leading to a time dependent boundary condition in a parabolic equation 
which is hard to deal with.  
That is why we  construct another transform which preserves normal derivative but is more sophisticated. Reduction of the PDEs to the fixed disk, with the help of aforementioned transform leads to new nonlinear terms, see $f_i$ and $g_i$, $i=1,2$ in \eqref{NEWactflow_in_terms_of_phi}--\eqref{skorost'}. These terms contain high order derivatives  and one needs to establish optimal regularity and decay results for linearized problem to employ fixed point argument for existence of solutions and  their stability. 
To this end we establish global regularity properties  for our free boundary problems (for general geometric regularity results in free boundary problems see \cite{CafSal2005}). 

Finally, we note that free boundary problems in cell motility  are closely related to the free boundary problems in tumor growth models. The key differences  are that in the letter models the  area  of  domain undergoes significant changes and  there is no persistent motion (see, e.g.,  \cite{Fri2004},
\cite{PerVau2015}, and  \cite{MelRoc2017}).  

 




The paper is organized as follows. In Section \ref{Section_the_model} we introduce a 2D model of active gel that is a free boundary problem with Keller-Segel PDEs. 
In Section \ref{section_lin_an_stst} we  consider linearization  around  radially symmetric steady states and introduce a function  of geometrical and physical parameters (the  domain radius, adhesion strength and myosin density).  Theorem \ref{thm_linearizeddisk} establishes a critical value of this function that separates  stability and instability regimes. 
In Section \ref{section_bifurcation} we show that at this critical value bifurcation of the steady states occurs and traveling wave solutions appear, as described in Theorem \ref{biftwtheorem}. These solutions model persistent motion which is the signature feature of  cytoskeleton gels  motility. Section \ref{section_lin_stab_tw} is devoted to linear stability analysis of the traveling wave solutions which yields stability up to a slow center manifold. Finally, Theorem \ref{nonlinear_stability} in Section \ref{section_nonlin_stab_stst} establishes nonlinear stability of steady states for subcritical values of the parameters.

{\bf Acknowledgments}. VR is grateful to PSU Center for Mathematics of Living and Mimetic Matter, and  to PSU Center for Interdisciplinary Mathematics  for support of his  two stays at Penn State. His travel was also supported by NSF grant DMS-1405769. We thank  our colleagues  I. Aronson, J.-F. Joanny, N. Meunier, A. Mogilner, J. Prost  and L. Truskinovsky for useful discussions and suggestions on the model. We also express our gratitude to the  members of the L. Berlyand's PSU  research team, R. Creese, J. King, M. Potomkin,  and A. Saftsten, for careful reading and help in  the preparation of the manuscript.  

\section{The model}
\label{Section_the_model}
 
We consider a 2D model of motion of  an active gel drop which occupies a domain $\Omega(t)$ with free boundary. 
The flow of the acto-myosin network  inside the domain $\Omega(t)$ is described by the velocity field $u$. In the adhesion 
dominated regime (overdamped flow) \cite{CalJonJoanPro2008}
, \cite{BlaCas2013}
$u$ obeys  the Darcy's law 
\begin{equation}
\label{DarcyLaw}
-\nabla p=\zeta u \quad \text{in}\ \Omega(t),
\end{equation}
where $-p$ stands for the scalar stress ($p$ is the pressure) and $\zeta$  is the constant effective
adhesion drag coefficient. 
We consider compressible gel 
(the actomyosin network is a compressible fluid,  incompressible cytoplasm fluid
can be squeezed easily into the dorsal direction in the cell \cite{NicNovPulRumBraSleMog2017}). The  main modeling assumption of this 
work
is the following constitutive law 
for the scalar stress $-p$
\begin{equation}
\label{constitutiveEQ}
-p= \mu\div u +k m-p_{\rm h},
\end{equation}
where $\mu\div u$ is the hydrodynamic stress ($\mu$ being the effective bulk 
viscosity of the gel), the middle term 
$k m$ is the active component of the stress 
which is proportional to the density 
$m=m(x,y,t)>0$ 
of myosin motors with a constant contractility coefficient $k>0$,  
$p_{\rm h}$ is the constant 
 homeostatic pressure. 
 Throughout this work we assume that 
the effective bulk viscosity and the contractility coefficient $k$ in \eqref{constitutiveEQ} are scaled to $\mu=1$, 
$k=1$. We prescribe the following condition on the boundary
\begin{equation}\label{Young_Laplace_eq}
p+p_{\rm e}=\gamma \kappa \quad \text{on} \ \partial \Omega(t),
\end{equation}
known as the Young-Laplace equation, where $\kappa$ denotes the curvature, $\gamma>0$ 
is a constant coefficient and $p_{\rm e}$ is the effective traction which describes the mechanism of 
approximate conservation of the area due to the membrane-cortex tension. The traction $p_{\rm e}$ generalizes  the one-dimensional nonlocal spring condition introduced in \cite{RecPutTru2013}, \cite{RecPutTru2015}, see a more recent work \cite{PutRecTru2018} which also introduces   the cell volume regulating homeostatic pressure
\footnote{The author are grateful to L.Truskinovski for bringing 
\cite{PutRecTru2018} to their attention 
and helpful discussions on bifurcations 
during the preparation of the manuscript.},  and  we  similarly assume the simple linear dependence $p_{\rm e}=-k_{\rm e} |\Omega|$ of $p_{\rm e}$ on the area $|\Omega|$, where $k_{\rm e}$ is the  stiffness coefficient.
The evolution of motor density is described by the advection-diffusion equation
\begin{equation}
\label{myosin_equations}
\partial_t m=\Delta m -\div (u m) \quad \text{in} \ \Omega(t)
\end{equation}
 and no flux boundary condition in moving domain
\begin{equation}
\label{myosin_bc}
\partial_\nu m =((u\cdot \nu) - V_\nu)m \quad\text{on}\ \partial\Omega(t),
\end{equation}
$\nu$ stands for the outward pointing normal vector and $V_\nu$ is the normal velocity of the domain $\Omega(t)$. 
Finally, we assume continuity of velocities on the boundary 
\begin{equation}
\label{KinematicBC}
V_\nu=(u\cdot\nu),
\end{equation} 
so that \eqref{myosin_bc} becomes the standard Neumann condition. Combining    \eqref{DarcyLaw}--\eqref{KinematicBC} yields a closed set of equations that forms a model of cell motility  investigated in this work.   

It is convenient to introduce  the potential for the velocity field $u$  using \eqref{DarcyLaw}:
$$
u= \nabla \phi=- \nabla \frac{1}{\zeta} p
$$
and  rewrite problem \eqref{DarcyLaw}--\eqref{KinematicBC} in the form
\begin{equation}
\label{actflow_in_terms_of_phi}
\Delta \phi +m=\zeta \phi +p_{\rm eff}(|\Omega(t)|) \quad\text{in}\ \Omega(t),
\end{equation}
\begin{equation}\label{actin_bc_potential}
\zeta\phi=-\gamma \kappa\quad\text{on}\ \partial\Omega(t),
\end{equation}
\begin{equation}\label{actin_bc_normal}
V_\nu=\partial_\nu \phi \quad\text{on}\ \partial\Omega(t),
\end{equation}
\begin{equation}
\label{myosin_equations_I}
\partial_t m= \Delta m -\div (m \nabla \phi ), \quad \text{in} \ \Omega(t),
\end{equation}
\begin{equation}
\label{myosin_bc_I}
\partial_\nu m =0 \quad\text{on}\ \partial\Omega(t),
\end{equation}
where we introduce  the notation $p_{\rm eff}:=p_{\rm h}+p_{\rm e}=p_{\rm h}-k_{\rm e} |\Omega(t)|$. We  assume that the area  $|\Omega(t)|$ is  such that 
\begin{equation}
\label{gelswelling}
p_{\rm eff}=p_{\rm eff}(|\Omega(t)|) >0.
\end{equation}
Moreover,  we consider the stiffness coefficient $k_{\rm e}$ to be sufficiently large so that
  it penalizes changes of the area. For instance, it prevents from shrinking of $\Omega$ to a point or from infinite expanding. Precise lower bound  on  $ k_{\rm e}=-p_{\rm eff}^{\prime}$ is given in \eqref{areashrinkingprevention}, see also Remark \ref{shrinking_expansion}.

\begin{rem}\label{evolution}
We view problem \eqref{actflow_in_terms_of_phi}--\eqref{myosin_bc_I}  as an evolution problem with respect to two unknowns $m(x,y, t)$ and  $\Omega (t)$, while the potential $\phi(x,y,t)$ is considered as an additional 
unknown function defining evolution of the free boundary. Indeed, for given $\Omega (t)$ and $m(x,y,t)$	the function $\phi(x,y,t)$ is obtained as the unique solution of the elliptic problem 
\eqref{actflow_in_terms_of_phi}--\eqref{actin_bc_potential}, and its normal derivative $\partial_\nu \phi$ 
defines normal velocity of the domain due to \eqref{actin_bc_normal}. Problem \eqref{actflow_in_terms_of_phi}--\eqref{myosin_bc_I} is 
supplied with initial conditions for $m$ and $\Omega$ and it is natural not to include the unknown $\phi$ 
into the phase space of this evolution problem but rather in the definition of the operator governing the semi-group in this phase space that defines the evolution of  $m$ and $\Omega$.
\end{rem}

	
In what follows we assume for simplicity  that problem \eqref{actflow_in_terms_of_phi}--\eqref{myosin_bc_I} is symmetric with respect to 
$x$-axis. Specifically we assume symmetry of the initial data, domain  $\Omega(0)$ and  $m(x,y, t=0)$ which is preserved for 
$t >0$.

\section{Linear stability analysis of radially symmetric steady states}
\label{section_lin_an_stst}

Problem \eqref{actflow_in_terms_of_phi}--\eqref{myosin_bc_I} possesses a family of radially symmetric solutions with
both $\phi$ and $m$ being constant. For a given radius $R>0$ the constant solution, 
$\phi=\phi_0$  and  $m=m_0$, is obtained from \eqref{actin_bc_potential} and \eqref{actflow_in_terms_of_phi} in the domain $\Omega(t)=B_R$ and it is verified by the direct substitution ($B_R$ is the disk with radius $R$): 

\begin{align}\label{steady_state}
\Omega=B_R,\quad m_0:=-\gamma/ R+p_{\rm eff}(\pi R^2),
\\ 
\notag
\phi_0=-\gamma/(\zeta R).
\end{align}

It is convenient to use polar coordinate system $(r,\varphi)$  whose origin is moving with the 
domain, 
\begin{equation}\label{def_of_Omega}
\Omega(t)=\{(x=r\cos\varphi+X_c(t), y=r\sin\varphi); 0\leq r<R+\rho(\varphi,t)\},
\end{equation}
where $X_c(t)$ is an approximation of $\tilde{X_c}(t)$, the $x$ coordinate of the center of mass of $\partial \Omega$, and  $\rho (\varphi, t)$ satisfies the following orthogonality condition that eliminates infinitesimal  shifts
\begin{equation}
\label{centeredDomAin}
\int_{-\pi}^{\pi}\rho (\varphi, t)  \cos\varphi d\varphi =0, \quad \text{ for all }  t >0.
\end{equation}
Indeed, formula \eqref{centeredDomAin} is a linearization of the the $x$ coordinate of the center of mass  $\tilde{X}_c(t)$ of $\partial \Omega$: 
\begin{eqnarray}\label{center_of_mass_cond}
0&=&\frac{1}{|\partial \Omega| }\int\limits_{\partial \Omega} x\,d\sigma -\tilde{X}_c\nonumber
\\&=&\frac{1}{|\partial \Omega| }\int_{\partial \Omega}(x-\tilde{X}_c)\,\text{d}\sigma\nonumber\\
&=&\frac{1}{|\partial \Omega| }\int_{-\pi}^{\pi}(R+\rho (\varphi, t))  \cos\varphi \sqrt{(R+\rho(\varphi,t))^2+(\rho^\prime_\varphi)^2}d\varphi\nonumber\\
&=& \frac{1}{\pi }\int_{-\pi}^{\pi}\rho(\varphi,t) \,\cos(\varphi) \,d\varphi+O(\rho^2).
\label{center_of_mass_linearized_cond}
\end{eqnarray}
 Here $\sigma$ denotes the arc length. 
%

 Linearizing problem 
\eqref{actflow_in_terms_of_phi}--\eqref{myosin_bc_I}
around the radially symmetric steady state (for $m_0$ from \eqref{steady_state} and $\Omega(t)=B_R$) we get the following system
\begin{equation}
\label{RadSymLinearized1}
\partial_t \rho+\dot X_c\cos\varphi=\partial_r \phi \quad \text{on}\ \partial B_R, 
\end{equation}
\begin{equation}
\label{ur_e_dlyacrtin}
\Delta \phi + m=\zeta \phi+p_{\rm eff}^\prime(\pi R^2)R\int_{-\pi}^{\pi}\rho(\varphi)d\varphi \quad\text{in}\ B_R,
\end{equation}
\begin{equation}
\label{krizna_linearized}
\phi=\frac{\gamma}{R^2\zeta}(\rho^{\prime\prime}+\rho) \quad \text{on}\ \partial B_R,
\end{equation}
\begin{equation}
\label{ur_e_dlya_m}
\partial_t {m}=\Delta {m} - m_0 \Delta {\phi} \quad\text{in}\ B_R, \quad \partial_r m=0\quad \text{on}\ \partial B_R,\end{equation}
the integral term in \eqref{ur_e_dlyacrtin} appears due to linearization of the term $p_{\rm eff} (|\Omega|)$ in \eqref{actflow_in_terms_of_phi}, 
$\rho^{\prime\prime}$ denotes $\partial_\varphi^2 \rho$.

 In operator form
system \eqref{RadSymLinearized1}--\eqref{ur_e_dlya_m} reads 
\begin{equation*}
\frac{d}{dt}U=\mathcal{A} U, 
\end{equation*} 
where $U=(m,\rho)$ and $\mathcal{A}$ is the following operator
\begin{equation}
\mathcal{A}:\left[\begin{array}{c}m\\\rho\end{array}\right]
\mapsto \left[\begin{array}{c}\Delta m- m_0 \Delta \phi\\\partial_r\phi-\frac{\cos \varphi}{\pi}\int_{-\pi}^{\pi}\partial_r\phi \cos \tilde{\varphi}\,d\tilde{\varphi}\end{array}\right],
\label{operator1}
\end{equation}
where $\phi$ solves the time independent problem \eqref{ur_e_dlyacrtin}--\eqref{krizna_linearized} for given $m$ and $\rho$. This operator is considered on pairs $U=(m,\rho)$ such that $m\in H^2(B_R)$ and 
$\partial_r m=0$ on $\partial B_R$, $\rho\in H^4(-\pi,\pi)$  and  $\rho$ is an even $2\pi$-periodic function.   The integral term in \eqref{operator1} appears  when the orthogonality condition \eqref{centeredDomAin} is applied to \eqref{RadSymLinearized1}.
The study of well posedness of the linearized system \eqref{RadSymLinearized1}--\eqref{ur_e_dlya_m}  and its stability  
amounts to the spectral analysis of the operator $\mathcal{A}$.

Observe that due to radial symmetry of operator $\mathcal{A}$ as well as its symmetry with respect to $x$-axis, all eigenvectors of $\mathcal{A}$ are of the form $m=\hat{m}(r)\cos (n\varphi)$ and $\rho=\hat{\rho}\cos(n\varphi)$ for integer $n\geq 0$, and $\phi$, the solution of \eqref{ur_e_dlyacrtin}--\eqref{krizna_linearized}, is of the similar form: $\phi=\hat{\phi}(r)\cos (n\varphi)$. The eigenvalue problem for operator $\mathcal{A}$ is:
\begin{eqnarray}
\lambda m&=&\Delta m - m_0 \Delta \phi\text{ in }B_R,\label{eig1}\\
\lambda \rho&=&(1-\delta_{n1})\partial_r \phi \text{ on }\partial B_R,\label{eig2}\\
-\Delta \phi+\zeta \phi&=& -m- 2\pi p_{\rm eff}^\prime(\pi R^2)R\,\hat{\rho}\,\delta_{n0} \text{ in }B_R,\label{eig3}\\
\phi &=& -\dfrac{\gamma (n^2-1)}{R^2\zeta} \hat{\rho} \text{ on }\partial B_R,\label{eig4}\\
\partial_r m &=&0\text{ on }\partial B_R,\label{eig5}
\end{eqnarray}
where $\delta_{nk}$ is the Kronecker delta.

\begin{rem}\label{zero_ev}
(i) The operator $\mathcal {A}$ has 
zero eigenvalue with an eigenvector $(m, \rho(\varphi))=(0,\cos\varphi)$.
The eigenspace corresponding to eigenvalue 0 represents infinitesimal shifts of the reference solution $m=m_0,\rho=0$ and $\Omega=B_R$. To see this, note that if $\Omega=B_R+\varepsilon (1,0)$ (i.e., $\Omega$ is $B_R$ shifted by $\varepsilon$ along $x$-axis), then in view of \eqref{def_of_Omega} $\rho(\varphi)=\varepsilon  \cos \varphi + o(\varepsilon)$ for small $\varepsilon$. Moreover, since problem \eqref{actflow_in_terms_of_phi}--\eqref{myosin_bc_I} is translational invariant, then any shift of the solution is also a solution.  However, $(0,\varepsilon\cos\varphi)$ are eigenfunctions of operator $\mathcal{A}$ obtained from the linearization of the original problem, these eigenfunctions correspond to infinitesimal shifts, not exact shifts. 

(ii) Yet another zero eigenvalue of the operator $\mathcal{A}$ is obtained by taking derivative of the family of steady states \eqref{steady_state} with respect to the parameter $R$.  The corresponding eigenvector is  
$m=\gamma/R+2\pi p_{\rm eff}^\prime(\pi R^2)R$, $\rho =1$. 
\end{rem}
While the two aforementioned eigenvectors corresponding to zero eigenvalue are trivially obtained by taking derivatives of families of steady states  in the parameters, the following Lemma describes all other possible eigenvectors corresponding to the zero eigenvalue.

\begin{lem}
\label{och_polezn_lemma} 
 For $\zeta\geq m_0$ the operator $\mathcal{A}$ has zero eigenvalue corresponding to a nonconstant $m$ if and only if $m=m_0 (\phi_1(r)-r)\cos\varphi$ and $\phi_1(r)$
	solves
\begin{equation}
\label{BifCond1}
\frac{1}{r}(r\phi^\prime_1(r))^\prime-\frac{1}{r^2} \phi_1(r) +(m_0-\zeta)\phi_1(r) =m_0 r\quad 0\leq r<R,
\quad \phi_1(0)=0, \,\phi_1(R)=0.
\end{equation}	
and 
\begin{equation}
\label{BifCond1bis}
\phi_1^\prime(R) =1.
\end{equation}	
\end{lem}	
\begin{proof} Set $\lambda =0 $ in \eqref{eig1}-\eqref{eig5} and consider all integer $n\geq 0$.

If $n=0$, then $m=\hat{m}(r)$, $\phi=\hat{\phi}(r)$ and $\rho=\hat{\rho}$, and \eqref{eig1}, \eqref{eig2} can be written as 
\begin{equation*}
\left\{\begin{array}{l}-\Delta (m-m_0\phi)=0,\\ ~~\partial_r(m-m_0\phi)|_{r=R}=0.\end{array}\right.
\end{equation*} 
This implies that $m-m_0\phi=C_1$. Substituting $m=C_1+m_0\phi$ in \eqref{eig1} one obtains
\begin{equation}\label{pr_n0}
\left\{\begin{array}{l}-\Delta \phi + (\zeta-m_0)\phi =C_2,\text{ in } B_R,\\~~\partial_r\phi=0, \text{ on }\partial B_R, \end{array}\right.
\end{equation}
where $C_2=C_1-2\pi p'_h(\pi R^2)R^3\zeta \gamma^{-1}$. Since $\zeta\geq m_0$ one deduces that $\phi$ is constant 
Then $m$ is also constant,  $m=C_1+m_0\phi$. 

If $n=1$, then \eqref{eig1} implies that $u(r):=\hat{m}(r)-m_0\hat{\phi}(r)$ satisfies the following equation:
\begin{equation*}
\frac{1}{r}(ru')'-\frac{1}{r^2}u=0,
\end{equation*}  	
therefore $u(r)=C_3r$. Thus, $\hat{m}=C_3r+m_0\hat{\phi}$. Substituting this representation for 
$\hat{m}$ into \eqref{eig3} we obtain,
\begin{equation}
\frac{1}{r}(r\hat{\phi}')'-\frac{1}{r^2}\hat{\phi}+(m_0-\zeta)\hat{\phi}=-C_3m_0r.
\end{equation}
From continuity of $\phi$ at the origin we obtain that $\hat{\phi}(0)=0$. From \eqref{eig5} we obtain that $\tilde{\phi}'(R)=-C_3$. Now taking $\phi_1(r):=\tilde{\phi}(r)/C_3$ we see that both \eqref{BifCond1} and \eqref{BifCond1bis} are satisfied.

If $n\geq 2$, we have 
$$
\frac{1}{r}(r\hat\phi^\prime(r))^\prime-\frac{n^2}{r^2} \hat\phi(r) +(m_0-\zeta)\hat\phi(r) =0\quad 0\leq r<R,
\quad \hat\phi(0)=0, \,\hat\phi^\prime(R)=0.
$$
The latter problem has only trivial solutions 
for $\zeta\geq m_0$. Then from \eqref{eig1} and \eqref{eig5} we deduce that $m=0$, while \eqref{eig4} yields $\hat\rho=0$.

Therefore, there exists a non-constant $m$, corresponding to the zero eigenvalue (that is, solution of \eqref{eig1}-\eqref{eig5} with 
$\lambda = 0$) only in the case $n=1$, and in this case $\tilde{m}=m_0(\phi_1(r)-r)$ with $\phi_1(r)$ solving both \eqref{BifCond1} and \eqref{BifCond1bis}.


\end{proof}

\begin{thm} 
\label{thm_linearizeddisk}
(Linear stability of steady states \eqref{steady_state}).  Assume that  the myosin density $m_0$ is bounded above by the third eigenvalue of the operator $-\Delta$ in $B_R$ with the Neumann boundary condition on $\partial B_R$, also assume that $\zeta>m_0$ and $p_{\rm eff}^\prime(\pi R^2)$ satisfies 
\begin{equation}
	\label{areashrinkingprevention}
 p_{\rm eff}^\prime(\pi R^2)\leq 
 -\Bigl(\gamma/R +2m_0 + \sqrt{2R\sqrt{\zeta}}m_0\Bigr)/(2\pi R^2). 
\end{equation}
Let $\phi_1$ be the solution of \eqref{BifCond1}.
Then 
\begin{itemize}
\item[(i)]  if $\phi_1^{\prime}(R)<1$, then the operator $\mathcal{A}$ has zero eigenvalue $\lambda=0$ of multiplicity two, other eigenvalues have negative real parts,
\item[(ii)]  if $\phi_1^{\prime}(R)=1$, then the operator $\mathcal{A}$ has zero eigenvalue $\lambda=0$ of multiplicity three, other eigenvalues have negative real parts, 
\item[(iii)]  if $\phi_1^{\prime}(R)>1$, then the operator $\mathcal{A}$ has a positive eigenvalue $\lambda>0$.
\end{itemize}
\end{thm}

 \begin{rem} It is well known that if linearized operator has zero eigenvalue, then linear spectral analysis 
is inconclusive for stability/instability of the underlying nonlinear system. As explained in Remark \ref{zero_ev}, operator $A$ always has  zero eigenvalue  with at least two eigenvectors (corresponding to infinitesimal shifts and the derivative of the family of steady states with respect to the radius.  In Theorem \ref{nonlinear_stability}, we  establish stability  in  the  case (i)  in Theorem \ref{thm_linearizeddisk} by showing that the  first eigenvector can be eliminated thanks to invariance of the problem  \eqref{actflow_in_terms_of_phi}--\eqref{myosin_bc_I} with respect to shifts and 
 projection of the solution of \eqref{actflow_in_terms_of_phi}--\eqref{myosin_bc_I} on the second eigenvector can be controlled due to conservation of myosin. In the case (iii) in Theorem \ref{thm_linearizeddisk} the linearized system is unstable implying instability of nonlinear system \eqref{actflow_in_terms_of_phi}--\eqref{myosin_bc_I}. 
\end{rem}

 \begin{rem}\label{shrinking_expansion}
 Zero (radially symmetric) mode  is responsible for  the expansion and shrinking of the domain.  The condition  \eqref{areashrinkingprevention} assures that  real part of the corresponding eigenvalue is negative, leading to stability with respect to infinitesimal  expansion and  shrinking.	 
 	\end{rem}



\begin{proof}  Thanks to radial symmetry of the problem (and our assumption about symmetry with respect to the $x$-axis) 
eigenvectors of $\mathcal{A}$ have the form $m=\hat m_n(r)\cos n\varphi$, $\rho=\hat\rho_n\cos n\varphi$. Consider first the case 
$n\geq 2$. In this case \eqref{ur_e_dlyacrtin} takes form $\Delta \phi =\zeta \phi-m$, then we have 
\begin{equation*}
\lambda m=\Delta m +m_0 m-m_0 \zeta \phi.
\end{equation*}
Multply this equation by the complex conjugate $\overline {m}$  of $m$ and integrate over $B_R$ to find 
\begin{equation}
\label{Vychisleniya_0}
\lambda \int_{B_R} |m|^2dxdy=-\int_{B_R}|\nabla m|^2dxdy +m_0 \int_{B_R} |m|^2dxdy
-m_0\zeta \int_{B_R} \phi \overline {m}dxdy.
\end{equation}
Now multiply the equation $\overline{m}=\zeta\overline{\phi}-\Delta \overline{\phi}$ by $m_0\zeta\phi$ and integrate over $B_R$ to obtain the following representation for the last term in \eqref{Vychisleniya_0}:
\begin{equation*}
m_0\zeta \int_{B_R} \phi \overline {m}dxdy=m_0\zeta^2 \int_{B_R} |\phi|^2 dxdy+m_0\zeta\int_{B_R}|\nabla \phi|^2dxdy-m_0\zeta\int \phi {\partial_r \overline{\phi}} d\sigma
\end{equation*}
Since $\partial_r \overline{\phi}=\overline{\lambda}\overline{\rho}$ and by virtue of \eqref{krizna_linearized}
$\overline \rho=\frac{R^2\zeta}{\gamma (1-n^2)}\overline\phi$,  equation \eqref{Vychisleniya_0} rewrites as
\begin{align*}
\lambda \int_{B_R} |m|^2dxdy +\overline{\lambda} \frac{m_0 R^2\zeta^2}{\gamma(n^2-1)}\int |\phi|^2d\sigma=&-\int_{B_R}|\nabla m|^2dxdy+m_0 \int_{B_R} |m|^2dxdy\\ 
&- m_0\zeta \int_{B_R} |\nabla \phi|^2dxdy-m_0\zeta^2 \int_{B_R} |\phi|^2dxdy.
\end{align*} 
Note that for $n\geq 2$ the function $\hat m_n(r)\cos n\varphi$ is orthogonal to the first and second eigenfunctions
of the operator $-\Delta$  in $B_R$ with the Neumann condition on $\partial B_R$, recall also that $m_0$ is bounded by the third eigenvalue. Then by Proposition~\ref{lem:ineq_for_m} we have  
\begin{equation*}
\int_{B_R}|\nabla m|^2dxdy-m_0 \int_{B_R} |m|^2dxdy\geq 0,
\end{equation*}
so that real part of $\lambda$ is negative.

Consider now the case $n=0$ which corresponds to radially symmetric eigenfunctions. Taking the derivative of steady states 
with respect to the parameter $R$ we obtain an eigenvector corresponding to zero eigenvalue. Let us show that other radially symmetric eigenvectors correspond to eigenvalues with negative real parts. It is convenient to change the unknown 
$\tilde\phi:= \phi+2\pi R\rho p_{\rm eff}^\prime(\pi R^2)/\zeta$, then in view of \eqref{krizna_linearized} we have 
$\tilde\phi=\rho ( \gamma/ R^2+2\pi Rp_{\rm eff}^\prime(\pi R^2))/\zeta$ which in turn leads to the boundary condition 
\begin{equation*}
\partial_r \tilde \phi= \frac{\lambda\zeta}{\gamma/R^2+2\pi R p_{\rm eff}^\prime(\pi R^2)}\tilde\phi.
\end{equation*}
Then arguing as above we obtain the following relation 
\begin{equation}
\label{DlinnayaOtsenkaNoPonyatnaya}
\begin{aligned}
\lambda \int_{B_R} |m|^2dxdy -\overline{\lambda} \frac{m_0 \zeta^2}{\gamma/R^2+2\pi R p_{\rm eff}^\prime(\pi R^2)}&\int_{\partial B_R} |\tilde\phi|^2d\sigma=-\int_{B_R}|\nabla m|^2dxdy+m_0 \int_{B_R} |m|^2dxdy\\ 
&- m_0\zeta \int_{B_R} |\nabla \tilde \phi|^2dxdy-m_0\zeta^2 \int_{B_R} |\tilde \phi|^2dxdy.
\end{aligned} 
\end{equation}
By Proposition \ref{lem:ineq_for_m} we have 
\begin{equation}
\label{odnavspomog_otsenKa}
\int_{B_R}|\nabla m|^2dxdy-m_0 \int_{B_R} |m|^2dxdy\geq -m_0 \pi R^2 |\langle m\rangle|^2
\end{equation}
because of the radial symmetry of $m$, where $\langle m\rangle$ denotes the mean value of $m$,
 $\langle m\rangle :=\frac{1}{\pi R^2}\int_{B_R} m dxdy$. Therefore real part of $\lambda$ is negative if 
$\langle m\rangle=0$. Thus we can normalize the eigenvector by setting 
\begin{equation}
\label{meanmyosinravnoodyn}
\langle m\rangle=1.
\end{equation} 
Assume also that $\lambda\not=0$. Then integrating the equation $\lambda m=\Delta m-m_0 \Delta \tilde\phi $ we find
\begin{equation*}
\langle m\rangle :=\frac{1}{\pi R^2}\int_{B_R} m dxdy=-\frac{m_0}{\lambda \pi R^2}\int_{\partial B_R} \partial_r \tilde\phi d\sigma=-\frac{2m_0 \zeta}{{\gamma}/R+2\pi R^2 p_{\rm eff}^\prime(\pi R^2)}\tilde\phi (R).
\end{equation*}
Integrating also the equation $\Delta \tilde \phi+m=\zeta\tilde \phi$ we have
\begin{equation}
\label{srednyayatildephi}
\zeta\langle\tilde \phi\rangle=\langle m \rangle +\frac{1}{ \pi R^2}\int_{\partial B_R} \partial_r \tilde\phi d\sigma=(1-\lambda/m_0)\langle m \rangle.
\end{equation}
It follows from \eqref{DlinnayaOtsenkaNoPonyatnaya}--\eqref{meanmyosinravnoodyn} that real part of $\lambda$ is negative 
if  we prove that 
\begin{equation}
\label{nashaSuperTsel}
m_0\pi R^2 -m_0\zeta^2\pi R^2|\langle\tilde \phi\rangle|^2- m_0\zeta \int_{B_R} |\nabla \tilde \phi|^2dxdy-m_0\zeta^2 \int_{B_R} |\tilde \phi-\langle\tilde\phi\rangle|^2dxdy<0.
\end{equation}
By \eqref{srednyayatildephi} and \eqref{meanmyosinravnoodyn} the second term in \eqref{nashaSuperTsel} equals 
$-m_0\pi R^2 |1-\lambda/m_0|^2$, while the last term admits the following lower bound
\begin{equation}
- m_0\zeta \int_{B_R} |\nabla \tilde \phi|^2dxdy-m_0\zeta^2 \int_{B_R} |\tilde \phi-\langle\tilde\phi\rangle|^2dxdy\leq 
-m_0\zeta Q|\tilde\phi(R)-\langle\tilde \phi\rangle|^2,
\end{equation}
where $Q$ is given by
\begin{equation}
\label{miniminiminimizat}
Q=\inf
\left\{\int_{B_R} |\nabla w|^2dxdy+\zeta\int_{B_R} |w|^2dxdy; \langle w \rangle=0, \,w(R)=1\right\}.
\end{equation}
Thus \eqref{nashaSuperTsel} is satisfied if the inequality 
$$
\pi R^2 |1-\lambda/m_0|^2+\frac{Q}{4m_0^2\zeta} |-\gamma/R-2\pi R^2 p_{\rm eff}^\prime(\pi R^2)-2m_0+2\lambda m_0|^2>\pi R^2
$$
holds for every $\lambda>0$, and this is true, in particular, if $-2\pi R^2 p_{\rm eff}^\prime(\pi R^2)\geq \gamma/R+2m_0+ 2\sqrt{\pi\zeta}R m_0/\sqrt{Q}$. The solution $Q$ of the minimization problem  \eqref{miniminiminimizat} is given by 
$$
Q=2\pi\zeta R^2\frac{ I_1(\sqrt{\zeta} R)}{R\sqrt{\zeta} I_2(\sqrt{\zeta} R)}.
$$
where $I_{1}$, $I_2$ are the modified Bessel functions of the first kind. Then using the bound 
$
Q\geq 2\pi \sqrt{\zeta} R
$
we arrive at the inequality from the hypothesis of the Theorem, $-2\pi R^2 p_{\rm eff}^\prime(\pi R^2)\geq \gamma/R+2m_0+ \sqrt{2R\sqrt{\zeta}}m_0$. Finally, if the eigenvalue $\lambda$ is zero, then \eqref{srednyayatildephi} yields 
$\zeta \langle \tilde\phi\rangle=\langle m\rangle$. We use this relation  in \eqref{odnavspomog_otsenKa} and substitute the result 
into \eqref{DlinnayaOtsenkaNoPonyatnaya} to find that $\tilde \phi$ is constant. This implies that $m$ is constant as well so that
this eigenfunction coinsides  with that obtained by taking derivative of steady states in the parameter $R$.

Consider now the case $n=1$. Introduce the space of functions $K_1=\{m\in H^1(B_R);\, m=\hat m(r) \cos\varphi\}$  and consider the quadratic form 
\begin{equation}
\label{formaquadratic}
F_\zeta [m]=\int_{B_R} |\nabla m|^2dxdy -m_0\int_{B_R} m^2dxdy+m_0\zeta\int_{B_R}|\nabla\phi|^2dxdy+m_0\zeta^2 \int_{B_R} \phi^2dxdy,
\end{equation}
where $\phi$ is the unique solution of the equation $\Delta \phi+m=\zeta\phi$ with the Dirichlet boundary condition 
$\phi=0$ on $\partial \Omega$.  Minimizing the Rayleigh quotient $F_\zeta [m]/\int_{B_R} m^2 dxdy$ on $K_1$  we obtain an eigenvalue $\lambda =-\min F_\zeta [m]/\int_{B_R} m^2 dxdy$. Indeed, a minimizer $m$ satisfies
$-\Delta m-m_0 m+m_0\zeta\phi=-\lambda m$ in $B_{R}$ and $\partial_r m=0$ on $\partial B_R$. Thus the pair $m$
and $\rho=0$ is an eigenvector corresponding to the eigenvalue $\lambda$. 

 Now to prove (iii) calculate 
$F_\zeta[m]$ with $m:=m_0(\phi_1(r)-r)\cos\varphi$. In this case $\phi=\phi_1(r)\cos\varphi$ and we have, integrating by parts,
\begin{equation*}
F_\zeta[m]=\int_{\partial B_R}m \partial_r m d\sigma +\int_{B_R}(-\Delta m+m_0 m+\zeta m_0\phi_1)m dxdy=\pi R^2 
m_0^2 (1-\phi_1^\prime(R))<0.
\end{equation*}
Thus the operator $\mathcal{A}$ has a positive eigenvalue.

 To prove  (i) observe that $-\min F_\zeta [m]/\int_{B_R} m^2 dxdy$ provides the exact upper bound for real parts of eigenvalues other than zero eigenvalue  which corresponds to infinitisimal  shifts (in fact one can see that eigenvalues for $n=1$ 
just coincide with those of the selfadjoint operator generated by the form $-F_\zeta [m]$).  Assume, by contradiction,  that  $F_\zeta [m]<0$ for some $m\in K_1$. Observe that 
$F_\zeta[m]$ continuously increases in $\zeta$ and $F_\zeta[m]\to+\infty$ as $\zeta\to+\infty$. Indeed, let $\hat\zeta>\zeta$ let  $\phi$ and $\hat\phi$  solve $\Delta \phi+m=\zeta \phi$ in $B_R$, $\phi=0$ on $\partial B_R$ 
and $\Delta\hat \phi+m=\hat\zeta  \hat\phi$ in $B_R$, $\hat\phi=0$ on $\partial B_R$, correspondingly. Introduce 
$\tilde \phi=\hat\zeta \hat\phi/\zeta$, then 
\begin{align*}
-\frac{\hat\zeta}{\zeta}\left(\int_{B_R}|\nabla \hat \phi|^2dxdy+\hat \zeta \int\hat \phi^2dxdy\right) &=-\frac{\zeta}{\hat \zeta}\int_{B_R}|\nabla \tilde \phi|^2dxdy-\zeta \int\tilde \phi^2dxdy\\
&=
\inf_{\overline{\overline{\phi}}}\left(\frac{\zeta}{\hat \zeta}\int_{B_R}|\nabla \overline{\overline{\phi}}|^2dxdy+\zeta \int|\overline{\overline{\phi}}|^2dxdy -2\int_{B_R}m\overline{\overline{\phi}} dxdy\right)\\
&\leq
\inf_{\overline{\overline{\phi}}}\left(\int_{B_R}|\nabla \overline{\overline{\phi}}|^2dxdy+\zeta \int |\overline{\overline{\phi}}|^2dxdy -2\int_{B_R}m\overline{\overline{\phi}} dxdy\right)\\
&=-\int_{B_R}|\nabla \phi|^2dxdy-\zeta \int  \phi^2dxdy.
\end{align*}
Next we show that there exists $\hat{\zeta}>\zeta$ such that 
\begin{equation*}
\min\limits_{m\in K_1}F_{\hat\zeta}[m]/\int_{B_R}m^2 \, dxdy = 0.
\end{equation*}
Assume by contradiction that there exists a sequence $\zeta_k\to \infty$ and $m_k\in K_1$ such that $\|m_k\|_{L^2(B_R)}=1$ and $F_{\zeta_k}[m_k]~<~0$. Then 
\begin{equation}\label{estimate_with_m_0}
\int_{B_R}|\nabla m_k|^2 \,dxdy+m_0\zeta\int_{B_R}\left(|\nabla \phi_k|^2+ \zeta \phi^2_k\right) \, dxdy<m_0,
\end{equation}
where $\Delta \phi_k+m_k=\zeta_k \phi_k$ in $B_R$, $\phi_k=0$ on $\partial B_R$. 
Observe that $\zeta_k\phi_k - m_k\rightharpoonup 0$ weakly in $L^2(B_R)$. Indeed, multiply equation $\Delta \phi_k+m_k=\zeta_k \phi_k$ by a test function $v\in H^1(B_R)$
\begin{equation*}
\langle \nabla \phi_k, \nabla v \rangle + \langle m_k-\zeta_k\phi_k,v \rangle =0 
\end{equation*}
and pass to the limit as $k\to \infty$ (note that $\|\nabla \phi_k\|<1/\sqrt{\zeta_k}$ by \eqref{estimate_with_m_0}). Thus, $m_k-\zeta_k \phi_k \rightharpoonup 0$ weakly  in  $L^2(B_R)$. On the other hand, due to \eqref{estimate_with_m_0}, $m_k$ is bounded in $H^1(B_R)$, so that there exists $m^{*}\in H^1(B_R)$ such that, up to a subsequence,  $m_k\to m_*$ strongly in $L^2(B_R)$, and thus 
$\liminf\limits_{k\to \infty} (\|\zeta_k \phi_k\|^2_{L^2(B_R)}-\|m_k\|^2_{L^2(B_R)} )\geq 0$. Then $F_{\zeta_k}[m_k]<0$ implies  that 
\begin{equation}
\int_{B_R}|\nabla m_k|^2 \, dxdy + m_0\left\{\int_{B_R}|\zeta_k\phi_k|^2\,dxdy - \int_{B_R}|m_k|^2 \, dxdy\right\}<0.
\end{equation}
By passing to the limit $k\to \infty$ we obtain that $\nabla m^*=0$ and thus $m^*\equiv \text{const}$, which obviously contradicts $\langle m_k \rangle =0$ (we consider case $n=1$) and $\|m_k\|_{L^2(B_R)}=1$.

Thus, $\min F_{\hat\zeta}[m]/\int_{B_R} m^2 dxdy=0$ for some $\hat \zeta >\zeta$.
Then by Lemma \ref{och_polezn_lemma}
the solution of 
\begin{equation}
\label{BifCond1aux}
\frac{1}{r}(r\hat\phi^\prime_1(r))^\prime-\frac{1}{r^2} \hat\phi_1(r) +(m_0-\hat\zeta)\hat \phi_1(r) =m_0 r\quad 0\leq r<R,
\quad\hat \phi_1(0)=0, \,\hat\phi_1(R)=0.
\end{equation}	
satisfies
\begin{equation}
\label{BifCond1bisaux}
\hat\phi_1^\prime(R) =1.
\end{equation}	
But  $-\frac{1}{r}(r(\hat\phi^\prime_1(r)-\phi^\prime_1(r)))^\prime+\frac{1}{r^2}( \hat\phi_1(r)-\phi_1(r)) +(\zeta-m_0)
(\hat \phi_1(r)-\phi_1(r)) =(\zeta-\hat\zeta)\hat\phi_1 >0$  for $0\leq r<R$, and $\hat \phi_1(0)-\phi_1(0)=\hat\phi_1(R)-\phi_1(R)=0$. By the maximum principle $\hat\phi_1(r)-\phi_1(r)> 0$ for $0< r<R$, therefore  
 $\hat\phi^\prime_1(R)\leq \phi^\prime_1(R)$, i.e. $\phi^\prime(R)\geq 1$, contradiction.

Finally (ii) follows by the uniqueness of the solution of \eqref{BifCond1}.
\end{proof}


\section{Bifurcation of traveling waves from the family  of steady states}
\label{section_bifurcation}

In this Section we show that zero eigenvalue corresponding to eigenvector described in Lemma \ref{och_polezn_lemma} leads to a 
bifurcation of traveling wave solutions from the family of radially symmetric steady states \eqref{steady_state} parametrized by $R$. 
This bifurcation is determined by three parameters: the size of the cell $R$,  and adhesion strength $\zeta$ which are independent parameters and  the myosin density $m_0$.  Due to zero force balance in the steady state, surface tension (determined by  curvature $R^{-1})$),  myosin contraction (determined by myosyn density $m_0$),  and  homeostatic pressure $p_{\rm eff} (\pi R^2)$  are in equilibrium, which provides the dependence between $m_0$ and $R$ given by the second equation in \eqref{steady_state}.  
It is convenient to choose $R$ as the bifurcation parameter in the bifurcation conditions \eqref{BifCond1}-\eqref{BifCond1bis}. 

Consider traveling wave  solutions moving with velocity $V>0$ in $x$-direction. Substitute the traveling wave ansatz 
\begin{equation}
\label{tw_ansatz}
m=m(x-Vt,y),\ \phi=\phi(x-Vt,y),\ \Omega(t)=\Omega+(Vt,0)
\end{equation}
 to \eqref{actflow_in_terms_of_phi}--\eqref{myosin_bc_I}  to derive stationary free boundary problem for the unknowns $\phi$ and $\Omega$ 
%
\begin{equation}
\label{tw_Liouvtypeeq}
\Delta \phi +\Lambda \frac{e^{\phi-Vx}}{\frac{1}{|\Omega|}\int_\Omega e^{\phi-Vx}dxdy} =\zeta \phi +p_{\rm eff}(|\Omega|)\quad \text{in}\ \Omega, \quad \partial_\nu(\phi-Vx)=0\quad \text{on}\ \partial\Omega,
\end{equation}
\begin{equation}
\label{addit_cond_tw}
\quad \zeta\phi=-\gamma \kappa\quad \text{on}\ \partial\Omega. 
\end{equation}
Indeed, \eqref{myosin_equations_I} yields $-V\partial_x m= \Delta m -\div (m \nabla \phi )$ in $\Omega$  while 
$\partial_\nu \phi=V \nu_x$ on $\partial \Omega$, then, taking into account the boundary condition $\partial_\nu m=0$, we see
that 
\begin{equation}
\label{myosin_density}
m=\Lambda e^{\phi-Vx}\Bigl/\Bigr.{\frac{1}{|\Omega|}\int_\Omega e^{\phi-Vx}dxdy}.
\end{equation}
 Here unknown positive constant    $\Lambda$ is a part of the solution (cf. spectral parameter). Integrating  \eqref{myosin_density} over $\Omega$ one sees that $\Lambda$ is the average myosin density.  For  convenience of  the analysis, we will   use the single parameter  $R$ related to the radius of the disk in steady states, via
setting $\Lambda=\Lambda(R) := p_{\rm eff}(\pi R^2)-\gamma/ R $ (c.f. \eqref{steady_state}).

%

\begin{thm} (bifurcation of traveling waves)
\label{biftwtheorem}
 Let $R_0$ be such that the solution of \eqref{BifCond1} with $R=R_0$ and 
$m_0=\Lambda(R_0)= p_{\rm eff}(\pi R_0^2)-\gamma/ R_0$ satisfies 
\eqref{BifCond1bis}. Assume also that $m_0<\zeta$, 
$p_{\rm eff}^\prime(\pi R^2_0)\leq -\gamma/(2\pi R^3_0)$ and 
\begin{equation}
\label{transversality_expl}
\frac{d}{dR}\left(\frac{\zeta I_1(R\sqrt{\zeta-\Lambda(R)})}{(\zeta-\Lambda(R))^{3/2}I_1^{\prime}(\sqrt{\zeta-\Lambda(R)})}-\frac{R\Lambda(R)}{\zeta-\Lambda(R)}\right)\Bigl.\Bigr|_{R=R_0}\not=0 . 
\end{equation}
Then there exists a family of solutions of  \eqref{tw_Liouvtypeeq}--\eqref{addit_cond_tw} 
parametrized by the velocity $V$. Moreover  if $|V|\leq V_0$ (for some $V_0>0$) then these solutions (both the function $\phi$ and the domain $\Omega$)  are smooth, depend smoothly on the parameter $V$.  When $V=0$ the solution is the radial steady state
$\Omega=B_{R_0}$, $m=m_0= p_{\rm eff}(\pi R_0^2)-\gamma/ R_0$.
\end{thm}

%

\begin{proof} 
As above we consider $\Omega$ in polar coordinates, $\Omega =\{0 \leq r< R +\rho(\varphi)\}$. Since $\zeta > \Lambda(R_0)$, for sufficienly small $\rho$, $V$ and $R$ sufficiently close to $R_0$ there is a unique solution  
$\Phi=\Phi(x,y; V,R, \rho)$ of \eqref{tw_Liouvtypeeq}.  It depends on three parameters:
 the  scalar parameter $V$ (the prescribed velocity), the radius $R$ via the parametrization of the domain and 
$\Lambda=\Lambda(R)$, and the functional parameter
$\rho$ that describes the shape of the domain $\Omega$ or, more precisely, its deviation from the disk $B_R$. As above we assume the symmetry of the domain with respect to the $x$-axis whose shapes are described by even functions $\rho$.
 
  The condition \eqref{addit_cond_tw} on the unknown boundary, described by $\rho(\varphi)$,   rewrites as 
\begin{equation}
\label{TW_equationontheboundary}
-\gamma\frac{(R+\rho)^2 + 2(\rho^\prime)^2-\rho^{\prime\prime}(R+\rho)}{((R+\rho)^2+(\rho^\prime)^2)^{3/2}}=\zeta\Phi ((R+\rho (\varphi))\cos\varphi, (R+\rho (\varphi))\sin\varphi, V, R, \rho).
\end{equation}
As before, to get rid of infinitesimal shifts we require \eqref{centeredDomAin}. Then introducing the 
function $\mathcal{B}$ which maps from $\mathcal{X}=C^{2,\alpha}_{\rm per}(-\pi,\pi)\times \mathbb{R}\times \mathbb{R}$ to 
$\mathcal{Y}=C^{0,\alpha}_{\rm per}(-\pi,\pi)\times \mathbb{R}$: 
\begin{equation}
\label{operator_equation}
\mathcal{B}(\rho, V;R):=\left(\gamma\frac{(R+\rho)^2 + 2(\rho^\prime)^2-\rho^{\prime\prime}(R+\rho)}
{\zeta((R+\rho)^2+(\rho^\prime)^2)^{3/2}}+\Phi, \int_{-\pi}^\pi \rho\cos\varphi d\varphi\right),
\end{equation}
we rewrite problem \eqref{tw_Liouvtypeeq}--\eqref{addit_cond_tw} in the form
\begin{equation}
\label{FuncAnTWequation}
\mathcal{B}(\rho,V;R)=0. 
\end{equation}
Next we apply the Crandall-Rabinowitz bifurcation theorem \cite{CraRab1971} (Theorem 1.7),
which guarantees bifurcation of new smooth branch of solutions provided that
%
\begin{itemize}
 		\item[(i)] $\mathcal{B}((\rho,V);R)=0$ for all $R$ in a neighborhood of $R_0$;
 		\item[(ii)] there exist continuous  $\partial_{(\rho,V)}\mathcal{B}$, $\partial_R\mathcal{B}$, and $\partial^2_{(\rho,V),R}\mathcal{B}$ in a neighborhood of $(\rho,V)=0$, $V=V_0$;  
 		\item[(iii)] $\mathrm{Null}(\partial_{(\rho,V)}\mathcal{B})$ and $\mathcal{Y}\backslash \mathrm{Range}(\partial_{(\rho,V)}\mathcal{B})$  at $(\rho,V)=0$, $R=R_0$ are one-dimensional;
 		\item[(iv)]  $\partial^2_{(\rho,V),R}\mathcal{B}(\rho,V)\notin  \mathrm{Range}(\partial_{(\rho,V)}\mathcal{B})$ at $(0,R=R_0)$ 
for all 
$(\rho,V)\in \mathrm{Null}(\partial_{(\rho,V)}\mathcal{B})$.
 	\end{itemize}
Condition (i) is satisfied.  Condition (ii) can be verified as in \cite{BerFuhRyb2018}.

 	To verify (iii), we begin by calculating $\mathcal{L}:=\partial_{(\rho,V)}\mathcal{B}$ at $0$.
 Linearizing  \eqref{operator_equation} around $\rho=0$, $V=0$ we get
\begin{equation}\label{linearization}
\mathcal{L}: (\rho, V)\mapsto 
\left(-\frac{\gamma}{R^2\zeta} (\rho^{\prime\prime}+\rho)+V\partial_V \Phi(R\cos\varphi,R\sin\varphi; 0,R,0)+\langle \partial_\rho \Phi,  \rho \rangle|_{V=0, \rho=0},
\int_{-\pi}^{\pi} \rho(\varphi) \cos \varphi d\varphi
\right)
\end{equation}
Here $\langle \partial_\rho \Phi,  \rho \rangle|_{V=0, \rho=0}$ denotes the Gateaux derivative of  $\Phi$ at $V=0$ and $\rho=0$.  
We have $\langle \partial_\rho \Phi,  \rho \rangle|_{V=0, \rho=0}=-\frac{R}{\zeta}p_{\rm eff}^\prime(\pi R^2)\int_{-\pi}^\pi\rho d\varphi$ and $\partial_V \Phi(R\cos\varphi,R\sin\varphi; 0,R,0)=\tilde\phi_1(R,R)\cos\varphi$, where $\tilde \phi_1(r,R)$ solves
\begin{equation}
\label{bicond_po_drugomu}
\frac{1}{r}(r\tilde\phi^\prime_1(r,R))^\prime-\frac{1}{r^2} \tilde\phi_1(r,R) +(\Lambda(R)-\zeta)\tilde\phi_1(r,R) =\Lambda(R) r\quad 0\leq r<R,
\quad \tilde\phi_1(0,R)=0, \,\tilde\phi_1^\prime (r,R)|_{r:=R}=1.
\end{equation}
Note that if $\tilde \phi_1(R,R)\not =0$ then operator $\mathcal{L}$ has a bounded inverse. In the case $\tilde \phi_1(R,R)=0$ for $R=R_0$
(when operator $\mathcal{A}$ has an eigenvector with non-constant density $m$, see Lemma \ref{och_polezn_lemma}) the kernel of the operator $\mathcal{L}$  is one-dimensional $(\rho=0, V=1)$ and its range consists of all the pairs $(f,C)$ such that 
$\int_{-\pi}^{\pi}f(\varphi)\cos\varphi d\varphi=0$. Thus, condition (iii) holds.

It remains to verify (iv). To this end, we check if  
$\partial_R \mathcal{L}|_{R:=R_0} (0,1)$ does not belong to the range of the opeartor $\mathcal{L}$ (transversality condition), where
$$
\partial_R \mathcal{L}|_{R:=R_0}: (\rho, V)\mapsto 
\left(\frac{2\gamma}{R^3_0\zeta} (\rho^{\prime\prime}+\rho)+V\frac{d}{d R}\tilde \phi_1(R,R)\Bigl.\Bigr|_{R=R_0}\cos\varphi
+\overline{C}(R_0)\int_{-\pi}^{\pi}\rho d\varphi
,0\right),
$$
where $\overline{C}(R_0)=-\frac{1}{\zeta}\left(p_{\rm eff}^\prime(\pi R_0^2)+2\pi R_0^2 p_{\rm eff}^{\prime\prime}(\pi R_0^2)\right)$. 
Thus the transversality condition reads
\begin{equation}
\label{transversalitycondition}
\Bigl.\frac{d}{d R}
\tilde\phi_1(R,R)\Bigr|_{r:=R_0}\not =0.
\end{equation} 
In order to check this condition  we change variable in \eqref{bicond_po_drugomu} by introducing $\psi(r,R):=\tilde\phi_1(Rr,R)$, 
this leads to the problem in the unit disk:
\begin{equation}
\frac{1}{r}(r\psi^\prime(r,R))^\prime-\frac{1}{r^2} \psi(r,R) +R^2(\Lambda(R)-\zeta)\psi(r,R) =R^3\Lambda(R) r\quad 0\leq r<1,
\quad \psi(0,R)=0, \,\psi^\prime(r,R)|_{r:=1}=R.
\end{equation}
The solution of this problem is given by 
$$
\psi(r,R)=-\frac{R\Lambda(R)}{\zeta-\Lambda(R)}r+\frac{\zeta I_1(R\sqrt{\zeta-\Lambda(R)}r)}
{(\zeta-\Lambda(R))^{3/2}I_1^{\prime}(\sqrt{\zeta-\Lambda(R)})},
$$
so that condition \eqref{transversalitycondition} writes as \eqref{transversality_expl}.
\end{proof}

\begin{rem} Introduce the following  function 
\begin{equation}
\label{FunktsiyaF}
F(R):=\frac{\zeta I_1(R\sqrt{\zeta-\Lambda(R)})}{(\zeta-\Lambda(R))^{3/2}I_1^{\prime}(\sqrt{\zeta-\Lambda(R)})}-\frac{R\Lambda(R)}{\zeta-\Lambda(R)}.
\end{equation}
 Then the condition \eqref{BifCond1bis} that selects $R$ in \eqref{BifCond1}
 (which is also the necessary bifurcation condition, cf. Theorem \ref{thm_linearizeddisk}, item (ii)) 
and the transversality condition \eqref{transversality_expl}  write as follows 
\begin{equation}
\label{VseVterminahF}
F(R_0)=0, \  F^\prime(R_0)\not =0.
\end{equation}
\end{rem}

Finally, we demonstrate qualitative agreement of our analytical results with  experimental results from \cite{VerSviBor1999} (crescent shape and concentration of myosin at the rear) by computing numerically the shape  and the distribution of myosin in the cell for traveling wave solutions with small velocities $V$.   Solutions are obtained via asymptotic expansions in small velocities $V$, similarly to Appendix in \cite{BerFuhRyb2018}, by substituting ansatz  $\phi=\phi_0+V \phi_1+V^2 \phi_2+....$, $\Omega=\left\{0\leq r \leq R_0+ V \rho_1(\varphi) + V^2 \rho_2(\varphi)+...\right\}$, $\Lambda=\Lambda_0+V \Lambda_1+V^2\Lambda_2...$ into \eqref{tw_Liouvtypeeq}-\eqref{addit_cond_tw}. Results are depicted in Figure~\ref{fig:john}.  
\begin{figure}[h]
	\begin{center}
		\includegraphics[width=0.85\textwidth]{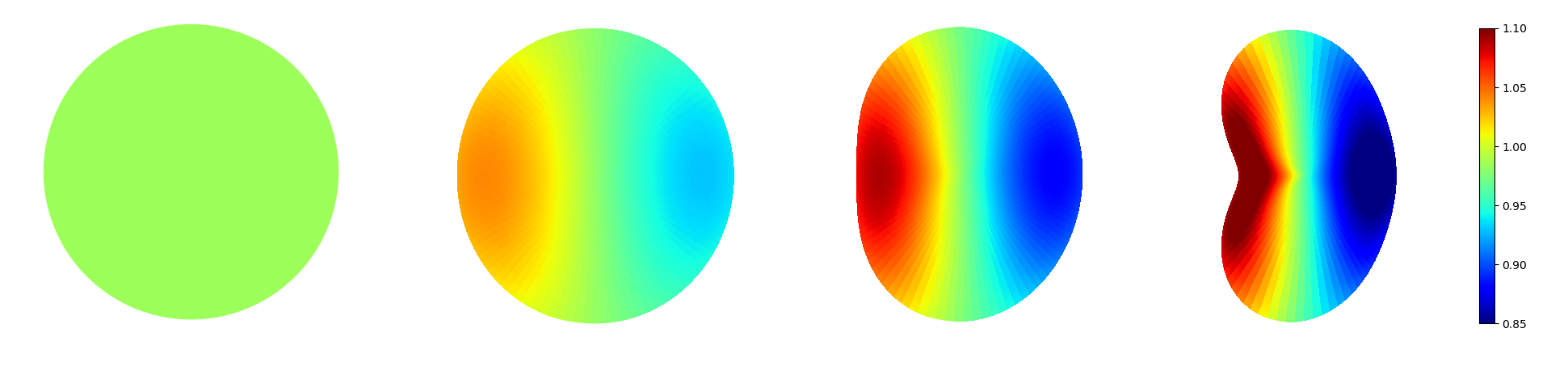}
		\caption{
			Approximate shape of traveling wave solutions for $m_0=3$, $\zeta=4$, $\gamma=0.03$, $V=0,\, 0.1,\,0.2,\,0.3$ 
			bifurcated from the radial steady
			state with $R_0=0.501$, which is a bifurcation value computed from  \eqref{BifCond1aux}--\eqref{BifCond1bisaux}.  The value $V=0$ corresponds to the circular shape, the higher $V$ is, the more pronounced  the crescent shape becomes.  The colors represent myosin density $m$: blue is for 
			lower $m$ and red is for higher $m$. \label{fig:john}
		}
	\end{center}
\end{figure}


\section{Linear stability analysis of traveling wave solutions}
\label{section_lin_stab_tw}

In this section we study linear stability of traveling wave solutions. We begin by writing down the system obtained after 
linearization of  \eqref{actflow_in_terms_of_phi}--\eqref{myosin_bc_I} around a traveling wave solution (cf. system  \eqref{RadSymLinearized1}--\eqref{ur_e_dlya_m} obtained by linearization of   \eqref{actflow_in_terms_of_phi}--\eqref{myosin_bc_I} around radial steady states).
The latter solution  is described by 
the domain $\Omega_{\rm tw}=\{0\leq \rho<R_0+\rho_{\rm tw}\}$, the potential $\phi=\Phi$ 
solving \eqref{tw_Liouvtypeeq}--\eqref{addit_cond_tw}, the myosin density $\tilde \Lambda e^{\Phi -Vx}$ with 
$\tilde\Lambda:=\Lambda |\Omega_{\rm tw}| /\int_{\Omega_{\rm tw}}  e^{\Phi -Vx}dxdy $, 
 and scalar velocity $V$ (the traveling wave solution is moving translationally in the $x$-direction).
 As before we assume the symmetry with respect to the $x$-axis of both traveling wave solution and its perturbations.
 Rewrite  \eqref{actflow_in_terms_of_phi}--\eqref{myosin_bc_I} in the system of coordinates moving with the traveling wave solution, i.e. introducing $x_{new}:=x_{old}-Vt$, and linearize around this solution, we have
\begin{equation}
\label{TwLinearized1}
\begin{aligned}
\frac{ (R_0+\rho_{\rm tw}) }{\sqrt{ (\rho_{\rm tw}^\prime)^2+(R_0+\rho_{\rm tw})^2}}\partial_t \rho=
\frac{\partial \phi}{\partial \nu}&+\rho\partial^2_{r\nu}({\Phi-Vx})+
\frac{\rho^\prime \sin \varphi +\rho\cos \varphi}{\sqrt{(\rho_{tw}^\prime)^2+(R_0+\rho_{tw})^2}}\frac{\partial}{\partial x}{(\Phi-Vx)}\\
&+\frac{-\rho^\prime \cos \varphi +\rho\sin \varphi}{\sqrt{(\rho_{tw}^\prime)^2+(R_0+\rho_{tw})^2}}\frac{\partial}{\partial y}{(\Phi-Vx)}\quad \text{on}\ \partial \Omega_{\rm tw}
\end{aligned}
\end{equation}
\begin{equation}
\label{Twnonovoeur_e}
\Delta \phi + m=\zeta \phi+p_{\rm eff}^\prime(\pi R^2)R\int_{-\pi}^{\pi}
(R_0+\rho_{tw})
\rho(\varphi)d\varphi \quad\text{in}\ \Omega_{\rm tw},
\end{equation}
\begin{equation}
\label{Twnovayakrizna_linearized}
\zeta(\phi+\rho\partial_r \Phi) =\kappa_{\rm tw}^\prime(\rho)\quad\text{on}\ \partial\Omega_{\rm tw},
\end{equation}
where
\begin{equation*}
\begin{aligned}
\kappa_{\rm tw}^\prime(\rho)=&
\frac{2(\rho_{\rm tw}+R)\rho-4\rho_{\rm tw}^\prime\rho^\prime-(\rho_{\rm tw}+R)\rho^\prime-\rho\rho_{\rm tw}^\prime }
{((\rho_{\rm tw}+R)^2+(\rho^\prime_{\rm tw})^2)^{3/2}}
\\
&-3
\frac{\rho(\rho_{\rm tw}+R)+\rho^\prime \rho_{\rm tw}^\prime}
{((\rho_{\rm tw}+R)^2+(\rho^\prime_{\rm tw})^2)^{5/2}} ((R+\rho_{\rm tw})^2+2(\rho_{\rm tw}^\prime)^2-(R+\rho_{\rm tw})\rho_{\rm tw}^\prime)
\end{aligned}
\end{equation*}

\begin{equation}
\label{Twmodified_ur_e_dlya_m}
\partial_t m=\Delta {m}+V\partial_x m -\div( \tilde \Lambda e^{\Phi-Vx} \nabla {\phi})
-\div(m\nabla \Phi) \quad\text{in}\  \Omega_{\rm tw}, 
\end{equation}

\begin{equation}
\label{Twposled_dynamic}
\rho\tilde \Lambda \partial^2_{r\nu}(e^{\Phi-Vx}) 
+\partial_\nu m+
\frac{\rho^\prime \sin \varphi +\rho\cos \varphi}{\sqrt{(\rho_{tw}^\prime)^2+(R+\rho_{tw})^2}}\frac{\partial}{\partial x}e^{\Phi-Vx}+\frac{-\rho^\prime \cos \varphi +\rho\sin \varphi}{\sqrt{(\rho_{tw}^\prime)^2+(R+\rho_{tw})^2}}\frac{\partial}{\partial y}e^{\Phi-Vx}=0\quad \text{on}\ \partial \Omega_{\rm tw}.
\end{equation}
This naturally leads to the following definition of the linearized operator:
\begin{equation}
\mathcal{A}_{\rm tw}:\left[\begin{array}{c}m\\\rho\end{array}\right]
\mapsto \left[\begin{array}{c}
\Delta {m}+V\partial_x m -\div( \tilde \Lambda e^{\Phi-Vx} \nabla {\phi})
-\div(m\nabla \Phi),\\
\text{right hand side of \eqref{TwLinearized1} }\times \frac{\sqrt{ (\rho_{\rm tw}^\prime)^2+(R_0+\rho_{\rm tw})^2}}
{R_0+\rho_{\rm tw}}\end{array}\right]
\label{Twoperator1}
\end{equation}

\begin{lem}
\label{3eigenvectors}
Let $\Phi=\Phi(x,y,V)$ and $\Omega_{\rm tw}=\{0<r<R_0+\rho_{tw}(\varphi, V)\}$ be solutions of \eqref{tw_Liouvtypeeq}--\eqref{addit_cond_tw} for $V\in (-V_0, V_0)$, and set $\tilde \Lambda:=\Lambda |\Omega_{\rm tw}| /\int_{\Omega_{\rm tw}}  e^{\Phi -Vx}dxdy$. Then the operator \eqref{Twoperator1} has zero eigenvalue of the  algebraic multiplicity (at least) three. The corresponding  eigenvectors are:

(i) the eigenvector  generated by infinitesimal shifts
\begin{equation}
\label{tw_shifts}
m_1=\tilde\Lambda \partial_x  e^{\Phi-Vx}, \quad \rho_1= \cos\varphi+\rho^\prime_{\rm tw}(\varphi)\frac{\sin\varphi}{R_0+\rho_{\rm tw}(\varphi)},
\end{equation}

(ii) the eigenvector linearly independent of  \eqref{tw_shifts} and emerging due to the total myosin mass conservation 
property,

(iii) there is also a  generalized eigenvector 
 \begin{equation}
 \label{TwderivativeinV}
 m_2= \partial_V (\tilde\Lambda e^{\Phi-Vx}), \quad \rho_2= \partial_V \rho_{\rm tw},
 \end{equation}
\end{lem}

\begin{proof} It is verified by straightforward  calculations that the pair $(m_1,\rho_1)$ given by \eqref{tw_shifts} satisfies 
$\mathcal{A}_{\rm tw}(m_1,\rho_1)=0$ and  $(m_2,\rho_2)$ given by \eqref{TwderivativeinV}  satisfies $\mathcal{A}_{\rm tw}(m_2,\rho_2)=(m_1,\rho_1)$.  Next we observe that  every solution of problem \eqref{TwLinearized1}--\eqref{Twposled_dynamic} satisfies the following  linearized version of the mass conservation property:
\begin{equation}
\label{Linearized_mass_preservation}
M(t):=\int_{\Omega_{\rm tw}} m dxdy +\int_{-\pi}^\pi (R+\rho_{\rm tw})\rho\tilde \Lambda e^{\Phi-Vx}d\varphi\ \ 
\text{is independent of}\ t,
\end{equation}
To explain \eqref{Linearized_mass_preservation}, we write a linear perturbation of the traveling wave solution as 
\begin{equation*}
m_\ve=\tilde\Lambda e^{\Phi-Vx}+\ve m, \quad \Omega_\ve=\{0\leq \rho<R_0+\rho_{\rm tw}+\ve\rho\}
 \end{equation*}
and note that 
$$
\begin{aligned}
\int_{\Omega_\ve} m_\ve dxdy -  \int_{\Omega_{tw}} \tilde\Lambda e^{\Phi-Vx} dxdy
&=
\ve\left[ \int_{\Omega_{\rm tw}} m dxdy +\lim_{\ve\to 0}\frac{1}{\ve}
\left\{\int_{\Omega_\ve}-\int_{\Omega_{\rm tw}}\right\}
\tilde \Lambda e^{\Phi-Vx}dxdy
\right]+O(\ve^2)\\
&=\ve M(t)+O(\ve^2).
\end{aligned}
$$ 
The property $\frac{d}{dt}M(t)=0$ is obtained by integrating \eqref{Twmodified_ur_e_dlya_m} over $\Omega_{\rm tw}$ and using \eqref{TwLinearized1}, \eqref{Twposled_dynamic}. In terms of the operator $\mathcal{A}_{\rm tw}$ this implies that 
the adjoint operator $\mathcal{A}_{\rm tw}^\ast$ has the eigenvector $m^\ast =1$, $\rho^\ast=\tilde\Lambda e^{\Phi-Vx}(R+\rho_{\rm tw})$. On the other hand it is not difficult to check that the Fredholm alternative can be applied to the operator 
$\mathcal{A}_{\rm tw}^\ast$ so that there is an eigenvector $(m_3,\rho_3)$ of the operator $\mathcal{A}_{\rm tw}$ which is not orthogonal to the eigenvector $(m^\ast,\rho^\ast)$ of $\mathcal{A}^\ast_{\rm tw}$ defined above. Next we note that 
$$
\int_{\Omega_{\rm tw}} m_1 m^\ast dxdy +\int_{-\pi}^\pi \rho_1 \rho^\ast d\varphi =
\int_{\Omega_{\rm tw}} \tilde\Lambda \partial_x  e^{\Phi-Vx} dxdy+\int_{\partial \Omega_{\rm tw}}\tilde\Lambda  e^{\Phi-Vx}\nu_x ds=0.
$$
Thus $(m_i,\rho_i)$, $i=1,2,3$ are linearly independent. 
\end{proof}

 For $V=0$ the structure of the spectrum of the operator 
$\mathcal{A}_{\rm tw}$ is described by Theorem \ref{thm_linearizeddisk}: it has zero eigenvalue of multiplicity three 
while other eigenvalues have negative real part. Next using Lemma \ref{3eigenvectors} by a perturbation argument we see that the   structure  of  the  spectrum  for small  but nonzero $V$ is essentially the same as for $V=0$.


\begin{thm}
Let $R_0$, $m_0(R_0):=-\gamma/ R_{0}+p_{\rm eff}(\pi R^2_{0})$ and $\zeta$  be as in Theorem \ref{thm_linearizeddisk},  i.e. 
$m_0$ does not exceed  the third eigenvalue of the operator $-\Delta$ in $B_{R_0}$ with the Neumann boundary condition on 
$\partial B_{R_0}$,  $\zeta\geq m_0$ and $p_{\rm eff}^\prime(\pi R^2_0)$ satisfies \eqref{areashrinkingprevention}.
 Assume also, that the bifurcation and transversality conditions \eqref{VseVterminahF} are satisfied. Then the linearized operator 
$\mathcal{A}_{\rm tw}$ around traveling waves  with sufficiently small velocities $V$ (described in Theorem \ref{biftwtheorem})  has zero eigenvalue with multiplicity three (see Lemma \ref{3eigenvectors}), other eigenvalues have negative real parts.
\end{thm}

Even though linear stability analysis of traveling waves  is inconclusive,  a valuable insight can be obtained from numerical simulations of the bifurcation from steady states to traveling waves. To identify the 
type of the bifurcation we expand solutions of the free boundary problem \eqref{tw_Liouvtypeeq}-\eqref{addit_cond_tw}    
for traveling waves in powers of the (small) 
velocities $V$. In particular, computing two terms of the expansions 
\begin{equation}
\label{ExPaNsIoNs}
\begin{aligned}
&\Omega=\{0\leq r<R_0+V^2\rho_2(\varphi)+\dots\},\quad \Phi=\phi_0+V \phi_1(r)\cos\varphi
+V^2\phi_2(r,\varphi)+\dots,
 \\
&\tilde{\Lambda}=m_0+V^2 e^{\phi_0}\tilde\Lambda_2+\dots,
\end{aligned}
\end{equation} 
where 
$\tilde{\Lambda}=|\Omega|\Lambda/\int_\Omega e^{\Phi-Vx}dxdy
$, amounts to solving problem 
\eqref{BifCond1}--\eqref{BifCond1bis} for 
the function $\phi_1(r)$ and also solving the following equation
\begin{equation}
\label{phi2}
\Delta\phi_2+m_0
\left(\phi_2+\frac{1}{2}(\phi_1-x)^2
\right)
=\zeta\phi_2-e^{\phi_0}\tilde{\Lambda}_2
+p_{\rm eff}^\prime(\pi R^2)R\int_0^{2\pi}\rho_2 d\varphi \quad\text{in}\ B_{R_0}
\end{equation}
with two boundary conditions
\begin{equation}
\label{vbc2}
\partial_r\phi_2=0\quad\text{on}\ \partial B_{R_0},
\end{equation}
and
\begin{equation}
\label{cbc2}
\zeta\phi_2=\frac{\gamma}{R^2}(\rho_2+\partial^2_\varphi\rho_2)
\quad\text{on}\ \partial B_{R_0}.
\end{equation}
Equation \eqref{phi2} and boundary conditions \eqref{vbc2}, \eqref{cbc2} are obtained by 
 substituting the expansions \eqref{ExPaNsIoNs} into 
\eqref{tw_Liouvtypeeq}-\eqref{addit_cond_tw} and collecting the terms of the order $V^2$, while the unknown constant 
$\tilde{\Lambda}_2$ is determined via the solvability condition which appears when considering terms of the order $V^3$.  
Next we present  the plot of the bifurcation picture based on numerics for first two terms in expansions  \eqref{ExPaNsIoNs}.  Since  the total myosin mass $M$ is an invariant 
for \eqref{actflow_in_terms_of_phi}--\eqref{myosin_bc_I} ($M$ is conserved in time) it is natural to choose it as the bifurcation parameter instead of $R$ when investigating the bifurcation. Figure \ref{fig:john1} depicts  dependence of the velocity of traveling waves on the 
total myosin mass 
$M=\tilde{\Lambda}\int_\Omega e^{\Phi-Vx}dxdy$, 
(including steady states when $V=0$). 
Observe that an increase of $M$ in the neighborhood of the bifurcation point $M_{cr}$  leads to a transition from stable to unstable steady states. Moreover, in some range of parameters
$R_0$ and $\gamma$ 
the myosin mass $M$ of the traveling wave solution first decreases with
the velocity $V$ then the graph bends and $M$  starts to increase. 
%
This numerical results \footnote{Numerical calculations depicted on figures Fig.\ref{fig:john} and Fig.\ref{fig:john1} were carried out by PSU students J.King and A.Safsten \textcolor{red}
	{who were supported from  }}
  suggest the following conjecture on stability/instability of the traveling waves, which reveals subcritical pitchfork bifurcation and will be rigorously proved in the upcoming work. 
  
\begin{figure}[h]
	\begin{center}
		\includegraphics[width=0.85\textwidth]{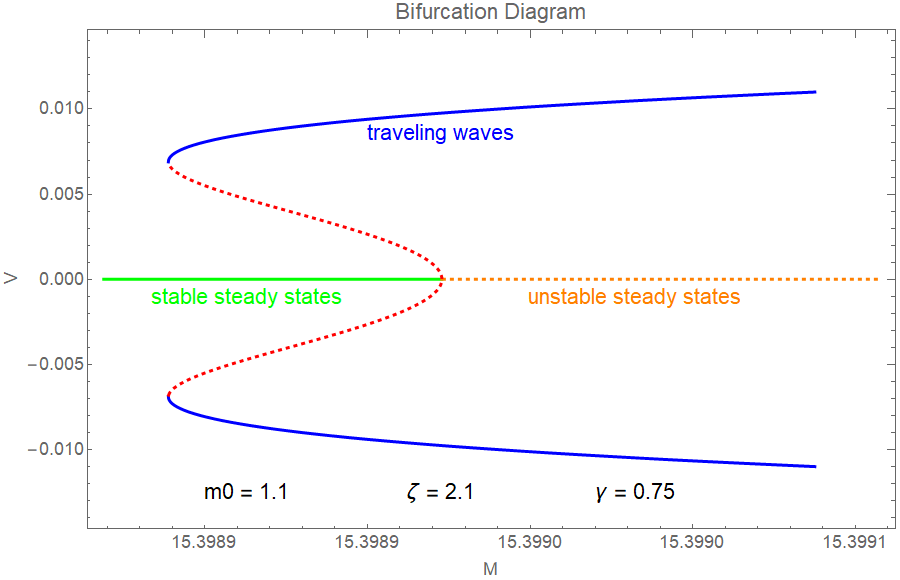}
		\caption{
			Bifurcation diagram,   $m_0=1.1$, $\zeta=2.1$, $\gamma=0.75$. \label{fig:john1}
		}
	\end{center}
\end{figure}

\begin{conj} 
In the range of the parameters such that the total myosin mass $M$ of traveling wave solutions  decreases in $V$
for small $V$, these solutions are nonlinearly 
unstable. 
\end{conj}

This Conjecture explains the nature of symmetry breaking at the onset of motion via instability of traveling for small velocities. Such a discontinuous transition from the rest to the 
steady motion has been observed in experiments and in direct numerical simulations of free boundary and phase field models 
\cite{BarLeeAllTheMog2015},\cite{ZieAra2015}. Note that the
plots of the asymptotic solutions of the model \eqref{actflow_in_terms_of_phi}--\eqref{myosin_bc_I} are obtained via rigorous analysis based on the Crandall-Rabinowitz bifurcation theorem and spectral analysis. In the one-dimensional model \cite{RecPutTru2013},\cite{RecPutTru2015} which is generalized in this work the supercritical pitchfork bifurcation to traveling waves is observed. This underscores the difference between the 2D model and its 1D prototype. Note that more sophisticated models in 1D also capture
subcritical bifurcation \cite{PutRecTru2018}.

\section{Nonlinear stability of radially symmetric steady states}
\label{section_nonlin_stab_stst}

As shown in Section \ref{section_lin_an_stst}, the linearized operator around radially symmetric steady states always has zero eigenvalue and therefore linear stability analysis is inconclusive for the nonlinear stability problem. Although Lyapunov function is not known in this problem,  we show that the invariant 
\begin{equation} \label{invariant}
\int_{\Omega(t)}m(x,y,t)dxdy
\end{equation}
 (total myosin mass) replaces Lyapunov function in the proof of nonlinear stability. This invariant corresponds to the eigenvector described in (ii) Remark \ref{zero_ev}  
 in the following sense. If the nonlinear problem has  such invariant, then the corresponding linearized problem also  has  analogous  invariant obtained by linearization of \eqref{invariant} in $\Omega(t)$ and $m(x,y,t)$. This linear invariant is the eigenvector of the adjoint linearized operator. Recall that the linearized operator has  another eigenvector   (see  (i) in Remark  \ref{zero_ev}) due to translational invariance of the problem.  In the stability analysis  below this eigenvector is taken into  account by the appropriate choice of the moving frame. 

Consider a radially symmetric  steady state with $R=R_0$ from the family \eqref{steady_state} and  assume that  $R_0$ 
is  such that   the following conditions hold 

(i)
\begin{equation}
\label{treteeigenval}
m_0:=-\gamma/ R_{0}+p_{\rm eff}(\pi R^2_{0}) \leq \lambda_3,  \quad \zeta>m_0, 
\end{equation}
where  $\lambda_3$ is the third eigenvalue of the operator $-\Delta$ in $B_{R_0}$ with the Neumann boundary condition on $\partial B_{R_0}$.

(ii)
 the hydrostatic pressure $p_{\rm eff}$ satisfies 
\begin{equation}
\label{hydrostat}
p_{\rm eff}^\prime(\pi R^2_0)< 
-\Bigl(\gamma/R_0 +2m_0 \Bigr)/(2\pi R^2_0), 
\end{equation}

(iii) \begin{equation}
\label{glavnyicondition}
\phi_1^{\prime}(R_0)<1,
\end{equation}
where  $\phi_1$ is the solution of 
\eqref{BifCond1} with $R=R_0$ (cf. Theorem \ref{thm_linearizeddisk}(i)).

\begin{thm}
\label{nonlinear_stability}
Let  radially symmetric steady  state \eqref{steady_state} with $R=R_0$ satisfy conditions  \eqref{treteeigenval}-\eqref{glavnyicondition}, then this steady state is stable in the following sense. 
	For any $\ve >0$ there exists $\delta (\ve) >0$ such that if the initial data 
	satisfies
	\begin{align}
	\label{firstvozmusch}
     &\Omega(0)=\{0\leq r<R_0+\delta \rho(\varphi)\} \quad \text{with}\ \|\rho_0\|_{H^{4}(-\pi,\pi)}<1,\\
     \label{secondvozmusch}
     &\|m(x,y,0)-m_0\|_{H^2(\Omega(0))}<\delta,
	\end{align}
$\frac{\partial}{\partial \nu} m\bigl|_{t=0}\bigr.=0$ on $\partial \Omega(0)$, and $\int_{\Omega(0)} m(x,y,0)dxdy=m_0 \pi R_0^2$,
then the solution $\Omega(t)$, $m(x,y,t)$ exists for all $t>0$ and satisfies
\begin{align}
\label{firstrezultat}
&\Omega(t)=\ve (X_c(t),0)+ \{0\leq r<R_0+\ve \rho(\varphi,t)\} \quad \text{with}\ \|\rho(\,\cdot,\,t)\|_{H^{4}(-\pi.\pi)}<1, 
|X_c(t)|\leq C,\\
\label{secondrezultat}
&\|m(x,y,t)-m_0\|_{H^2( \Omega(t))}<\ve,
\end{align}
where $\ve (X_c(t), 0) $ is shifted location of the linearized center of mass of $\partial \Omega(t)$ defined in \eqref{centeredDomAin}.
Moreover $ \|m(x,y,t)-m_0\|_{H^2( \Omega(t))}\to 0$, $ \|\rho(\,\cdot,\,t)\|_{H^{4}(-\pi.\pi)}\to 0$ as $t\to\infty$.
\end{thm}

\begin{proof} One can show by using Hille-Yosida theorem that the operator $\mathcal{A}$ is a generator of the $C_0$-semigroup 
$e^{\mathcal{A}t}U$ in the space $(m,\rho)=:U\in L^2(B_{R_0})\times (H^{1}_{\rm per}(-\pi,\pi)\setminus\{\cos\varphi\})$. 
As in Theorem \ref{thm_linearizeddisk}
introduce $\tilde\phi:= \phi+ {R_0}\frac{ p_{\rm eff}^\prime(\pi {R_0}^2)}{\zeta}\int_{-\pi}^{\pi}\rho d\varphi$. 
The operator $\lambda I-\mathcal {A}$ is defined for $\lambda>0$ 
via the bilinear form on $H^1(B_{R_0})\times(H^{7/2}_{\rm per}(-\pi,\pi)\setminus \{\cos\varphi\})$
$$
\begin{aligned}
(f,g)=(\lambda I -\mathcal{A})(m,\rho)\ \Longleftrightarrow\ &\int_{B_{R_0}}\nabla m\cdot \nabla \mu dxdy+(\lambda -m_0)\int_{B_{R_0}} m\mu dxdy\\
& +  \int_{B_{R_0}}\nabla \tilde\phi\cdot \nabla \tilde\psi dxdy+\zeta \int_{B_{R_0}}\tilde \phi \tilde \psi dxdy+\frac{\lambda \gamma}{R_0 \zeta}\int_{-\pi}^{\pi}(\rho^\prime\varrho^\prime-\tilde\rho\tilde\varrho) d\varphi
\\
&+ \zeta m_0\int_{B_{R_0}}\tilde \phi \mu dxdy
-\int_{B_{R_0}}m\tilde \psi dxdy\\
&
=\int_{B_{R_0}}f\mu dxdy+\frac{\lambda \gamma}{R_0 \zeta}\int_{-\pi}^{\pi}(g^\prime\tilde\varrho^\prime-g\tilde\varrho) d\varphi, \\  &\quad\forall \mu\in H^1(B_{R_0}), \ 
\varrho \in H^{7/2}_{\rm per}(-\pi,\pi)\setminus \{\cos\varphi\},
\end{aligned}
$$
where $\tilde \rho=\rho+{R_0^3}\frac{ p_{\rm eff}^\prime(\pi {R_0}^2)}{\gamma}\int_{-\pi}^{\pi}\rho d\varphi$,
$\tilde \varrho=\varrho+{R_0^3}\frac{ p_{\rm eff}^\prime(\pi {R_0}^2)}{\gamma}\int_{-\pi}^{\pi}\varrho d\varphi$, 
$\tilde\phi= \phi+ {R_0}\frac{ p_{\rm eff}^\prime(\pi {R_0}^2)}{\zeta}\int_{-\pi}^{\pi}\rho d\varphi$ and $\phi$ solves \eqref{ur_e_dlyacrtin}--\eqref{krizna_linearized}, and $\tilde\psi= \psi+ {R_0}\frac{ p_{\rm eff}^\prime(\pi {R_0}^2)}{\zeta}\int_{-\pi}^{\pi}\varrho d\varphi$, $\psi$ solving \eqref{ur_e_dlyacrtin}--\eqref{krizna_linearized} with $\mu$ in place of $m$ and 
$\varrho$ in place of $\rho$.

In order to proceed with the proof of nonlinear stability in Theorem~\ref{nonlinear_stability} we first show the regularity and exponential decay of the semigroup $e^{\mathcal{A}t}$ generated by linearized operator $\mathcal{A}$.
\begin{lem}
\label{Dlinnaya_i_pechalnaya_lemma}
 (regularity and decay properties  of solutions of the linearized problem \eqref{RadSymLinearized1}--\eqref{ur_e_dlya_m}) 
Under the conditions of Theorem \ref{nonlinear_stability}, the semigroup $e^{\mathcal{A}t}$ (where $\mathcal{A}$ is defined in \eqref{operator1}) has the following properties:
\begin{itemize}
\item[(i)] (decay property) For any initial data $U(0)=(m(x,y,0),\rho(\varphi,0))\in L^2(B_{R_0})\times (H^{1}_{\rm per}(-\pi,\pi)\setminus\{\cos\varphi\})$ the solution $U(t)=e^{\mathcal{A}t}U(0)$ of system \eqref{RadSymLinearized1}--\eqref{ur_e_dlya_m} is represented as 
\begin{equation}
\label{NashaTsel'}
\begin{aligned}
U(t)=c U_1/\Pi +\tilde U(t),
\end{aligned}
\end{equation}
where $c=\int_{B_{R_0}}m(\tilde x,\tilde y,0) d\tilde xd\tilde y+m_0\int_{-\pi}^\pi \rho(\tilde \varphi,0) {R_0} d\tilde \varphi$, 
$U_1=((\gamma/{R_0}+
2\pi p_{\rm eff}^\prime(\pi {R_0}^2){R_0}, 1)$,  is the eigenvector of the operator $\mathcal{A}$ corresponding to zero eigenvalue, see  Remark \ref{zero_ev},
$\Pi=\pi {R_0}^2 (\gamma/R_0+
2\pi p_{\rm eff}^\prime(\pi {R_0}^2){R_0})+2\pi {R_0} m_0$, and
\begin{equation}\label{exp_decay}
\begin{aligned}
\|\tilde U(t)\|_{ L^2(B_{R_0})\times (H^{1}(-\pi,\pi))}\leq C e^{-\theta t}  \| U(0)\|_{ L^2(B_{R_0})\times (H^{1}(-\pi,\pi))}
\end{aligned}
\end{equation}
with some constants $\theta>0$, $C$. Moreover, for $t>1$ estimate \eqref{exp_decay} improves to 
\begin{equation}\label{exp_decay_bis}
\begin{aligned}
\|\tilde U(t)\|_{ H^2(B_{R_0})\times (H^{4}(-\pi,\pi))}\leq C_1 e^{-\theta t}  \| U(0)\|_{ L^2(B_{R_0})\times (H^{1}(-\pi,\pi))}.
\end{aligned}
\end{equation}

\item[(ii)] (regularization) If $U(0)\in H^2_N(B_{R_0})\times( H^4_{\rm per}(-\pi,\pi)\setminus\{\cos\varphi\})$ then representation \eqref{NashaTsel'} holds with 
$\tilde U\in L^2((0,+\infty);  H^2_N(B_{R_0})\times 
(H^{11/2}_{\rm per}(-\pi,\pi)\setminus\{\cos\varphi\})$ and 
$$
\int_0^\infty \|\tilde U(t)\|^2_{ H^2(B_{R_0})\times 
H^{11/2}(-\pi,\pi)}dt+
\int_0^\infty \|\frac{d}{dt} \tilde U(t)\|^2_{ L^2(B_{R_0})\times 
H^{5/2}(-\pi,\pi)}dt
\leq C  \|\tilde U(0)\|^2_{ H^2(B_{R_0})\times  H^{4}(-\pi,\pi)}.
$$
\item[(iii)] For any $T>0$ and $F(t)\in L^2((0,T);  L^2(B_{R_0})\times 
(H^{5/2}_{\rm per}(-\pi,\pi)\setminus\{\cos\varphi\}))$ the solution $U(t)=\int_0^te^{A(t-\tau)}F(\tau)d\tau$ of the 
Cauchy problem  $\frac{d}{dt}U(t)=\mathcal{A}U(t)+F(t)$, $U(0)=0$ belongs to 
$$
L^2((0,T); H^2_{N}(B_{R_0})\times 
H^{11/2}(-\pi,\pi))\cap H^1((0,T); L^2(B_{R_0})\times 
H^{5/2}(-\pi,\pi))
$$ 
 and satisfies 
\begin{equation*}
\begin{aligned}
\int_0^T \| U(t)\|^2_{ H^2(B_{R_0})\times 
H^{11/2}(-\pi,\pi)}dt& + \int_0^T\Bigl\|\frac{d}{dt} U(t)\Bigr\|^2_{ L^2(B_{R_0})\times 
H^{5/2}(-\pi,\pi)}dt\\
&
\leq C(1+ T ^2) \int_0^T \| F(t)\|^2_{ L^2(B_{R_0})\times 
H^{5/2}(-\pi,\pi)}dt,\quad \forall 0\leq \tau\leq T,
\end{aligned}
\end{equation*}
where $C$ is independent of $T$.
\end{itemize}
\end{lem}

\begin{rem}
\label{remarkadlinnaya}
		Statement (i) establishes exponential stability of the linearized problem \eqref{RadSymLinearized1}--\eqref{ur_e_dlya_m} up to the constant eigenvector $U_1$. Here constant $c$ is the linearized total myosin mass. Indeed, if $m_0+\ve m_\ve$ is a perturbation of the steady state myosin density, then the total myosin mass expands as 
	\begin{eqnarray}
	&&\int_{\Omega_\ve}(m_0+ \ve m_\ve(x,y,0))dxdy\nonumber\\
	&&\hspace{40pt}=\int_{B_R}m_0 dxdy +\ve \left(\int_{B_R} m_\ve dxdy +\frac{1}{\ve}\left(\int_{\Omega_\ve\setminus B_R}-\int_{B_R\setminus \Omega_\ve}\right)m_0 dxdy  \right)+ O(\ve^2)\nonumber\\
	&&\hspace{40pt}=\int_{B_R}m_0 dxdy + \ve \left(\mathop{\underbrace{\int_{B_R} m_\ve dxdy +m_0R_0\int_{-\pi}^{\pi}\rho(\tilde\varphi,0)d\tilde\varphi}}_{\mathrm{linearized}\,\mathrm{total}\,\mathrm{myosin}\,\mathrm{mass}} \right)+ O(\ve^2). \nonumber
	\end{eqnarray}
	Constant $\Pi$ is chosen such that if one substitutes $U(t)\equiv U_1$ into \eqref{NashaTsel'}, then \eqref{NashaTsel'} becomes a trivial equality $U_1=cU_1/\Pi$ with $c=c(U_1)=\Pi$. 
	Constants $c$ and $\Pi$ can also be written as a projection in terms of dot-products:
 	\begin{equation}
	c=(U(0)\cdot\left[\begin{array}{c}1\\m_0\end{array}\right])_{L^2(B_R\times \partial B_R)}\quad\text{ and }\quad
	\Pi=(U_1\cdot\left[\begin{array}{c}1\\m_0\end{array}\right])_{L^2(B_R\times \partial B_R)}. 
	\end{equation}
	Representation \eqref{NashaTsel'} combined with the estimate \eqref{exp_decay} show that  time-dependent part $\tilde{U}(t)$ of the solution $U(t)$ is exponentially decaying in time, that is \eqref{exp_decay} establishes contraction property of the corresponding semi-group for sufficiently large time.

	 Statement (ii) establishes stability and regularity in stronger norms provided that initial conditions are sufficiently smooth. To explain the powers in (ii), note that $m$ belongs at least in $H^2(B_R)$ (so that LHS of \eqref{ur_e_dlya_m} in $L^2(B_R)$). Then from \eqref{ur_e_dlyacrtin} it follows that $\nabla \phi \in H^{5/2}(\partial B_R)$. Next, if one differentiates \eqref{krizna_linearized} in $\varphi$, then it follows that $\rho\in H^{3+5/2=11/2}(\partial B_R)$.

	  Statement (iii) is about the linearized problem if inhomogeneity ${F}(t)$ is added. This result is needed to extend stability of linearized problem to the nonlinear one by representing original problem $U_t=\mathcal{L}(U)$ as $U_t=\mathcal{A}U  +F(t)$ with nonlinearity $F(t)=\mathcal{L}(U)-\mathcal{A}U$.   
	\end{rem}

\begin{proof}
We employ Fourier analysis, representing $U=e^{\mathcal{A}t} U(0)$ as
\begin{equation}
\label{Fourier_ryd1}
U=\sum_{n=0}^\infty (\hat m_n(r,t), \hat \rho_n(t))\cos n\varphi,
\end{equation} 
then each pair $(\hat m_n(r,t), \hat \rho_n(t))\cos n\varphi$ satisfies system \eqref{RadSymLinearized1}--\eqref{ur_e_dlya_m} 
with $\phi=\hat\phi_n(r,t) \cos n\varphi$ solving for $n\geq 1$ the equation $\Delta (\hat\phi_n\cos n\varphi)+\hat m_n\cos n\varphi=
\zeta\hat\phi_n\cos n\varphi$ with the boundary condition $\hat \phi_n=\frac{\gamma}{R^2\zeta}(1-n^2)\hat\rho_n(t) \cos n\varphi$ on  $\partial B_R$. In the case $n=0$ it is convenient to seek $\phi$ in the form   $\phi =\tilde\phi_0(r,t)-2\pi {R_0}\hat \rho_0(t) p_{\rm eff}^\prime(\pi {R_0}^2)/\zeta$ then $\Delta \tilde\phi_0 +\hat m_0=\zeta \tilde\phi_0$ in $B_R$ and 
\begin{equation}
\label{vspomog_rvo1}
\tilde\phi_0(R_0,t)=
\Bigl[\frac{\gamma}{R^2\zeta}+2\pi {R_0}\hat  p_{\rm eff}^\prime(\pi {R_0}^2)/\zeta\Bigr]\hat\rho_0(t)
\end{equation}

Let us prove first the exponential stabilization of the zero mode. To this end integrate the equations $\Delta \tilde\phi_0 +\hat m_0=\zeta \tilde\phi_0$ and $\partial_t \hat m_0=\Delta \hat m_0-m_0\Delta \tilde\phi_0$ over $B_{R_0}$ to obtain
\begin{equation}
\label{perv_ssootn_rad_sym}
\frac{2}{R_0}\frac{d}{d t} \hat \rho_0+\langle\hat  m_0 \rangle=\zeta \langle \tilde \phi_0 \rangle,\quad 
\frac{d}{dt} \langle \hat m_0\rangle+\frac{2 m_0}{R_0}\frac{d}{dt}\hat\rho_0=0.
\end{equation}
The second equation (linearized myosin mass preservation) implies that $\langle\hat  m_0\rangle+\frac{2 m_0}{R_0}\hat \rho_0=M_0$ 
is conserved in time, therefore  the first equation in \eqref{perv_ssootn_rad_sym} rewrites  with the help of \eqref{vspomog_rvo1} as
\begin{equation}
\label{vtor_ssootn_rad_sym}
\frac{2}{R_0}\frac{d}{dt} \hat\rho_0= \left( \frac{2 m_0}{R_0}+ \frac{\gamma}{{R_0}^2}+2\pi {R_0} p_{\rm eff}^\prime(\pi {R_0}^2)\right) \hat\rho_0+ \zeta (\langle \tilde \phi_0 \rangle-\tilde \phi(R_0,t)) -M_0.
\end{equation}
Subtracting $cU_1/\Pi$ from the solution $U$ we reduce the study to the case $M_0=0$, see Remark \ref{remarkadlinnaya}, hence
\begin{equation}
\label{esche_formulka}
\frac{2}{R_0}\frac{d}{dt} \hat\rho_0=-\theta_2\hat\rho_0+ \zeta \langle \tilde \phi_0 \rangle,
\end{equation}
where $-\theta_2:=\frac{2 m_0}{R_0}+ \frac{\gamma}{{R_0}^2}+2\pi {R_0} p_{\rm eff}^\prime(\pi {R_0}^2)<0$.

Next multiply the equation
$\partial_t \hat m_0=\Delta \hat m_0-m_0\Delta \tilde\phi_0$   
 by $\tilde m=\hat m_0-\langle 
\hat m_0\rangle$ and integrate over $B_{R_0}$:
\begin{equation}
\label{tret_ssootn_rad_sym}
\frac{d}{2dt}\int_{B_{R_0}}\tilde m^2 dxdy= -\int_{B_{R_0}}|\nabla \tilde m|^2 dxdy+m_0 \int_{B_{R_0}}\tilde m^2 dxdy-\zeta m_0 \int_{B_{R_0}}\tilde \phi_0\tilde m  dxdy.
\end{equation}
Then we multiply the equation $\Delta \tilde \phi_0 +\hat m_0=\zeta\tilde \phi_0$ by $\tilde \phi_0 -\langle \tilde \phi_0\rangle$ and integrate over 
$B_{R_0}$:
\begin{equation*}
\int_{B_{R_0}}\tilde \phi_0\tilde m  dxdy=\zeta \int_{B_{R_0}} (\tilde \phi_0 -\langle \tilde \phi_0\rangle)^2 dxdy
+\int_{B_{R_0}}|\nabla (\tilde \phi_0 -\langle \tilde \phi_0  \rangle)|^2 dxdy
-2\pi R_0 \frac{d \hat \rho_0}{dt}(\tilde \phi_0(R_0,t)-\langle \tilde \phi_0 \rangle).
\end{equation*}
We use this equality in \eqref{tret_ssootn_rad_sym} to get 
\begin{equation}
\label{chetv_ssootn_rad_sym}
\begin{aligned}
\frac{d}{2dt}\int_{B_{R_0}}\tilde m^2 dxdy=& -\int_{B_{R_0}}|\nabla \tilde m|^2 dxdy+m_0 \int_{B_{R_0}}\tilde m^2 dxdy
-\zeta m_0 \int_{B_{R_0}}|\nabla (\tilde \phi_0 -\langle \tilde \phi_0\rangle)|^2 dxdy
\\&-\zeta^2 m_0 \int_{B_{R_0}} (\tilde \phi_0 -\langle \tilde \phi_0\rangle)^2 dxdy+
2\pi \zeta m_0 R_0 \frac{d \hat\rho_0}{dt}(\tilde \phi_0(R_0,t)-\langle \tilde \phi_0 \rangle). 
\end{aligned}
\end{equation}
Then  using 
\eqref{vtor_ssootn_rad_sym}  and the inequality $\int_{B_R}|\nabla (\tilde \phi_0 -\langle \tilde \phi_0\rangle)|^2dxdy+\zeta
\int_{B_R}(\tilde \phi_0 -\langle \tilde \phi_0\rangle)^2 dxdy\geq 
Q(\tilde \phi_0(R,t) -\langle \tilde \phi_0\rangle)^2$ (see 
\eqref{miniminiminimizat})
we get 
\begin{equation}
\label{pyat_ssootn_rad_sym}
\begin{aligned}
\frac{d}{2dt}\int_{B_{R_0}}\tilde m^2 dxdy\leq& -\int_{B_R}|\nabla \tilde m|^2 dxdy+m_0 \int_{B_{R_0}}\tilde m^2 dxdy-
\zeta m_0 Q(\tilde \phi({R_0}) -\langle \tilde \phi\rangle)^2
\\&
-4\pi m_0\left(\frac{d \hat \rho_0}{dt}\right)^2 +
\pi   m_0 R_0 \left( \frac{2 m_0}{R_0}+ \frac{\gamma}{R^2}+2\pi R p_{\rm eff}^\prime(\pi R^2_0)\right)\frac{d }{dt}\hat\rho_0^2 . 
\end{aligned}
\end{equation}
Since $m_0$ is less than the third eigenvalue of the Neumann Laplacian, by Proposition \ref{lem:ineq_for_m}  we have 
$\int_{B_R}|\nabla \tilde m|^2 dxdy-m_0 \int_{B_R}\tilde m^2 dxdy\geq \theta_1  \int_{B_R}\tilde m^2 dxdy$ for some 
$\theta_1>0$.
Then using 
\eqref{vtor_ssootn_rad_sym} once more we get
\begin{equation}
\label{shest_ssootn_rad_sym}
\begin{aligned}
\frac{d}{2dt}\int_{B_{R_0}}\tilde m^2 dxdy+(\pi m_0 R_0+2Q m_0/(\zeta R_0))\theta_2 \frac{d }{dt}\hat\rho_0^2 \leq-\theta_1\int_{B_R}|\tilde m|^2 dxdy-\theta_3\hat\rho_0^2,
\end{aligned}
\end{equation}
$\theta_3=Q\theta_2^2/\zeta^2>0$, this yields exponential decay of $\|\hat m_0\|_{L^2}$ and $|\hat\rho_0|$ as $t\to+\infty$.


Exponential decay of other modes is more simple to show (as in Theorem \ref{thm_linearizeddisk}).
For the component $n=1$
we have 
 $\hat \rho_1=0$ for all $t\geq 0$, then using positive definiteness of the form 
\eqref{formaquadratic} we get $\|\hat m_1 \|_{L^2}^2\leq e^{-\theta_4 t}\|\hat m_1\|^2|_{t=0}$, $\theta_4>0$. For higher harmonics, $n\geq 2$, we write $\partial_t \hat m_n\cos n\varphi=\Delta (\hat m_n\cos n\varphi)-m_0\Delta (\hat\phi_n\cos n\varphi)=
\Delta  (\hat m_n\cos n\varphi)+m_0 \hat m_n\cos n\varphi-\zeta m_0\hat\phi_n\cos n\varphi$ multiply by $\hat m_n\cos n\varphi$ 
and  integrate over $B_{R_0}$ to obtain, using the equality
$\hat m_n\cos n\varphi=\zeta \hat\phi_n\cos n\varphi-\Delta  (\hat\phi_n\cos n\varphi)$  and boundary conditions 
$\partial_r\hat\phi_n(R_0)=\frac{d\hat\rho_n}{dt}$, $\hat\phi_n(R_0)=-\frac{\gamma(n^2-1)}{R^2\zeta}\hat\rho_n$ ,
\begin{equation}
\label{higher _harm_ssootn_rad_sym}
\begin{aligned}
\frac{d}{4dt}\int_{B_{R_0}} \hat m^2_ndxdy=& -\int_{B_{R_0}}|\nabla (\hat  m_n\cos n\varphi)|^2 dxdy+\frac{m_0}{2} 
\int_{B_{R_0}} \hat m^2_n dxdy \\
&-
\zeta m_0 \int_{B_{R_0}}|\nabla (\hat \phi_n\cos n\varphi)|^2 dxdy
-\frac{\zeta^2 m_0}{2}  \int_{B_{R_0}} \hat \phi^2_n dxdy+\pi  \frac{m_0\gamma (1-n^2)}{2R_0}\frac{d}{dt}\hat\rho^2_n(t),
\end{aligned}
\end{equation}
where $\rho(\varphi,t)=\hat\rho(t) \cos n\varphi$.  Observe that for every function $\phi(r)$ one has
\begin{equation}
\label{lower_bound_harm}
\int_{B_{R_0}}|\nabla (\phi (r)\cos n\varphi)|^2 dxdy\geq |\phi(R_0)|^2\int_{B_{R_0}}|\nabla \left((r/R_0)^n\cos n\varphi\right)|^2 dxdy= \pi n|\phi(R_0)|^2.
\end{equation}
Plugging this bound into \eqref{higher _harm_ssootn_rad_sym} and applying Gronwall's  lemma we obtain 
\begin{align}
\label{gronwall1}
(\|\hat m_n(r,t)\|_{L^2(B_{R_0})}^2+n^2 |\hat\rho_n(t)|^2)\leq e^{-\overline\theta_1(n) t}(\|\hat m_n(r,0)\|_{L^2(B_{R_0})}^2
+n^2|\hat \rho_n(0)|^2)\quad \text{with}\   \overline\theta_1(n)\geq c n^2,
\end{align}
where $c>0$. This proves  \eqref{exp_decay}. Also estimates \eqref{gronwall1} yield  \eqref{exp_decay_bis} via a bootstrap procedure described in the proof of (iii).

Now we proceed with item (iii). Represent $F(t)=(f(r,\varphi,t), g(\varphi,t))$ as $F(t)=s(t)U_1/\Pi+\tilde F$,
 where $s(t)=\int_{B_{R_0}}fdxdy+m_0R_0 \int_{-\pi}^{\pi} g d\varphi$ and $\tilde F=(\tilde f,\tilde g )$ satisfies 
$\int_{B_{R_0}}\tilde f dxdy+m_0R_0 \int_{-\pi}^{\pi} \tilde g d\varphi=0$ for all $t$.  Then 
$$
U(t)=U_1\int_0^t s(\tau)d\tau/\Pi+\tilde U, \quad \tilde U=(\tilde m(r,\varphi,t),\tilde \rho(\varphi,t))
$$ 
and by 
item (i) we have  
$$
\|\tilde U(t)\|_{L^2(B_{R_0})\times H^1(-\pi,\pi)}\leq u(t):=
C\int_0^te^{-\theta (t-\tau)}\|F(\tau)\|_{L^2(B_{R_0})\times H^1(-\pi,\pi)}d\tau. 
$$
Since $$\frac{d u}{dt}u +\theta u^2=C u(t) \|F(t)\|_{L^2(B_{R_0})\times H^1(-\pi,\pi)} \leq \frac{C^2}{2\theta}\|F(t)\|^2_{L^2(B_{R_0})\times H^1(-\pi,\pi)}+\frac{\theta}{2} u^2, $$
 it holds that
 $$
\int_0^T u^2dt \leq \frac{C^2}{\theta^2}\int_0^T\|F(t)\|^2_{L^2(B_{R_0})\times H^1(-\pi,\pi)}dt.
$$
In particular, for every Fourier component $\tilde\rho_n(t)$ of $\tilde\rho=\sum \tilde \rho_n(t) \cos n \varphi$ we have 
\begin{equation}
\label{Fourier_comp_tilde_rho}
 \int_0^T \tilde\rho_n^2(t) dt \leq C\int_0^T\|F(t)\|^2_{L^2(B_{R_0})\times H^1(-\pi,\pi)}dt.
\end{equation}
To improve \eqref{Fourier_comp_tilde_rho} for $n>2$ expand $f$, $g$, $\tilde m$ and $\tilde\phi$ into Fourier series
$f=\sum f_n(r,t)\cos n\varphi$, $g=\sum g_n(t)\cos n\varphi$, $\tilde m=\sum \tilde m_n(r,t)\cos n\varphi$, $\tilde \phi=\sum \tilde\phi_n(r,t)\cos n\varphi$., where $\tilde\phi$ is the solution of problem \eqref{tilde_phi_eqn}.   Then, arguing as in the derivation of \eqref{higher _harm_ssootn_rad_sym} we get for $n>2$
\begin{equation}
\label{higher _harm_ssootn_rad_sym_bis}
\begin{aligned}
\frac{d}{4dt}\int_{B_R} \tilde m^2_ndxdy=& -\int_{B_R}|\nabla (\tilde m_n\cos n\varphi)|^2 dxdy+\frac{m_0}{2} \int_{B_R} \tilde m^2_n dxdy+\frac{1}{2} \int_{B_R} f_n \tilde m_n dxdy \\
&-
\zeta m_0 \int_{B_R}|\nabla (\tilde \phi_n\cos n\varphi)|^2 dxdy
-\frac{\zeta^2 m_0}{2}  \int_{B_R} \tilde \phi^2_n dxdy\\
&+\pi  \frac{m_0\gamma (1-n^2)}{2R_0}\frac{d}{dt}\tilde\rho^2_n(t)-\pi  \frac{m_0\gamma (1-n^2)}{R_0}g_n(t)\tilde\rho_n(t),
\end{aligned}
\end{equation}
Now we use here the bound \eqref{lower_bound_harm}, integrate the result  from $0$ to $T$ in time to obtain that
\begin{equation}
\label{Fourier_comp_tilde_rho_bis}
cn^5 \int_0^T \tilde \rho_n^2(t)dt \leq C \int_0^T\Bigl(\int_{B_{R_0}} f_n^2dxdy+n^2g_n^2(t)\Bigr)dt,\quad n>2,
\end{equation}
where $c>0$ and $C$ are independent of $n$, $t$ and $T$. Thus \eqref{Fourier_comp_tilde_rho} and \eqref{Fourier_comp_tilde_rho_bis}
imply that $\|\rho\|_{L^2(0,T;H^{5/2}(-\pi,\pi))}\leq C\|F\|_{L^2(0,T;B_{R_0})\times H^1(-\pi,\pi))}$. Then, by elliptic estimates applied to
$-\Delta\tilde\phi+\zeta\tilde\phi=\tilde m$ in $B_{R_0}$ with the boundary 
condition $\tilde\phi=\frac{\gamma}{R^2\zeta}(\tilde\rho^{\prime\prime}+\tilde\rho)$ on $\partial B_R$ we have 
$\int_0^T\|\tilde\phi\|^2_{H^2(B_{R_0})}dt \leq C\int_0^T\|F(t)\|^2_{L^2(B_{R_0})\times H^1(-\pi,\pi)}dt$.
This allows us to improve bound for $\tilde m$, applying parabolic estimates to the equation $\partial_t \tilde m-\Delta \tilde m+\tilde m=(m_0+1) \tilde m -\zeta m_0 \tilde \phi$
(where we consider the right hand side as known) with the boundary condition $\partial_r \tilde m=0$ on $\partial B_{R_0}$. We find that 
\begin{equation}
\label{m_my_otsenili}
\begin{aligned}
\int_0^T (\|\tilde m\|^2_{H^{2}(B_{R_0})}+\|\partial_t \tilde m\|^2_{L^2(B_{R_0})}) dt &\leq 
C\int_0^T\|(m_0+1)\tilde m -\zeta m_0\tilde \phi\|^2_{L^2(B_{R_0})}dt\\
&\leq  C\int_0^T\|F(t)\|^2_{L^2(B_{R_0})\times H^1(-\pi,\pi)}dt.
\end{aligned}
\end{equation}
%

We also improve  bounds \eqref{Fourier_comp_tilde_rho_bis} for $n>2$. To this end represent the solution $\tilde\phi$ of
\begin{equation}
\label{tilde_phi_eqn}
\Delta \tilde\phi+\tilde m=\zeta\tilde \phi\quad\text{in}\ B_{R_0},
\quad \tilde \phi=\frac{\gamma}{R^2_0\zeta}(\tilde\rho^{\prime\prime}+\tilde \rho) \quad \text{on}\ \partial B_{R_0}
\end{equation}
as $\tilde\phi=\tilde\phi^{(1)}+\tilde\phi^{(2)}$, where
\begin{equation*}
\begin{aligned}
\Delta \tilde\phi^{(1)}=\zeta\tilde \phi^{(1)}\quad\text{in}\ B_{R_0},
\quad\tilde \phi^{(1)}=\frac{\gamma}{R^2\zeta}(\tilde\rho^{\prime\prime}+\tilde \rho) \quad \text{on}\ \partial B_{R_0},\\
\Delta \tilde\phi^{(2)}+\tilde m=\zeta\tilde \phi^{(2)}\quad\text{in}\ B_{R_0},
 \quad\tilde\phi^{(2)}=0 \quad \text{on}\ \partial B_{R_0}.
\end{aligned}
\end{equation*}
Next expand $\tilde \phi^{(1)}$ and $\tilde\phi^{(2)}$ into the Fourier series 
$\tilde\phi^{(1)} =\sum \tilde\phi_n^{(1)}(r,t)\cos n\varphi$, 
$\tilde\phi^{(2)}=\sum \tilde\phi_n^{(2)}(r,t)\cos n\varphi$ and  multiply \eqref{tilde_phi_eqn} 
by $\tilde\phi_n(r,t)\cos n\varphi$, $n\geq 2$, to find, integrating over $B_{R_0}$
\begin{equation}
\label{eqn83}
\begin{aligned}
\int_{B_{R_0}}(|\nabla (\tilde\phi_n\cos n\varphi)|^2+\frac{1}{2}(\zeta\tilde\phi_n^2-m\tilde\phi_n))dxdy&=\int_{-\pi}^{\pi}\tilde \phi_n \partial_r\tilde \phi_n R_0\cos^2n\varphi
d\varphi\\
&=\frac{\pi \gamma}{R_0\zeta}(1-n^2)\tilde\rho_n(t)\left(\frac{d\tilde\rho_n(t)}{dt}-g_n(t)\right).
\end{aligned}
\end{equation}
On the other hand, the left hand side of \eqref{eqn83} rewrites as 
$$
\int_{B_{R_0}}\left(|\nabla (\tilde\phi_n^{(1)}\cos n\varphi)|^2+\frac{\zeta}{2}(\tilde\phi_n^{(1)})^2\right)dxdy+
\frac{\gamma\pi}{{R_0}\zeta}(1-n^2)\tilde\rho_n(t)\partial_r\tilde\phi_n^{(2)}(R_0,t),
$$
and using \eqref{lower_bound_harm} we obtain
\begin{equation}
\label{Nudostalouzhe}
c n^5\tilde\rho^2_n(t)+\frac{\pi \gamma}{2R_0\zeta}(n^2-1)\frac{d\tilde\rho_n^2(t)}{dt}\leq 
\frac{C}{n}\left(\bigl(\partial_r\tilde\phi_n^{(2)}(R_0,t)\bigr)^2+g_n^2\right),
\end{equation}
with $c>0$ and $C$ independent of $n$. Now multiply \eqref{Nudostalouzhe} by $n^6$ integrate in $t$ from $0$ to $T$ and add up 
the inequalities obtained to find that
\begin{equation}
\label{Nusovsemnadoelo}
\int_0^T\|\tilde\rho\|_{H^{11/2}(-\pi,\pi)}^2dt \leq C_1  \int_0^T\|\partial_r\tilde\phi^{(2)}\|_{H^{5/2}(-\pi,\pi)}^2dt+C_2 \int_0^T\|F(t)\|^2_{L^2(B_{R_0})\times H^{5/2}(-\pi,\pi)}dt. 
\end{equation}
It remains to note that by elliptic estimates that $\|\tilde\phi^{(2)}\|_{H^4(B_{R_0})}\leq C \|\tilde m\|_{H^2(B_{R_0})}$, which yields
$\|\partial_r\tilde\phi^{(2)}\|_{H^{5/2}(-\pi,\pi)}\leq C \|\tilde m\|_{H^2(B_{R_0})}$, and exploit \eqref{m_my_otsenili} to obtain the required bound for $\|\tilde\rho\|_{H^{11/2}(-\pi,\pi)}$ in $L^2(0,T)$. Also, since $\partial_t \tilde\rho =\partial_r \tilde \phi^{(1)}+\partial_r \tilde\phi_2 $ and
$\|\tilde \phi^{(1)}\|_{H^4(B_{R_0})}\leq C \|\tilde\rho\|_{H^{11/2}(-\pi,\pi)}$ we have $\int_0^T\|\partial_t \tilde\rho\|^2_{H^{5/2}(-\pi,\pi)}dt\leq C_2 \int_0^T\|F(t)\|^2_{L^2(B_{R_0})\times H^{5/2}(-\pi,\pi)}dt$.   

To prove (ii) we first obtain from \eqref{gronwall1} the following bound for the $\rho$-component $\tilde \rho$ of $\tilde U$, 
\begin{equation}
\int_0^\infty \|\tilde \rho\|^2_{H^{3}(-\pi,\pi)}dt \leq C  \|\tilde U(0)\|^2_{ H^1(B_{R_0})\times  H^{2}(-\pi,\pi)}.
\end{equation}
By (i) we also know that $\|\tilde m\|_{L^2(B_{R_0})}\leq C e^{-\theta t}  \| U(0)\|_{ L^2(B_{R_0})\times (H^{1}(-\pi,\pi))}$, 
therefore, arguing as in item (iii) one can show that  
\begin{equation}
\label{esche_boot_strap}
\int_0^\infty (\|\tilde m\|^2_{H^{2}(B_{R_0})}+\|\partial_t \tilde m\|^2_{L^2(B_{R_0})}) dt \leq 
C\int_0^\infty \|(m_0+1)\tilde m -\zeta m_0\tilde \phi\|^2_{L^2(B_{R_0})}
\leq C  \|\tilde U(0)\|^2_{ H^1(B_{R_0})\times  H^{2}(-\pi,\pi)}.
\end{equation}
Following further the lines of the  proof of item (iii) we eventually get  
\begin{equation}
\label{Nusovsemnadoelo_konets}
\int_0^\infty\|\tilde\rho\|_{H^{11/2}(-\pi,\pi)}^2dt \leq C_1  \int_0^\infty\|\partial_r\tilde\phi^{(2)}\|_{H^{5/2}(-\pi,\pi)}^2dt+
C_2 \|\rho\|^2_{H^{4}(-\pi,\pi)}\bigl|_{t=0}. 
\end{equation}
Then again arguing as in item (iii) we complete the proof of the Lemma.
\end{proof}

\begin{cor}
\label{cor_pro_l_infty_bounds} Under conditions of Theorem \ref{nonlinear_stability} the following uniform in $t\in[0,T]$ bounds hold
$$
\|e^{\mathcal{A}t}U_0\|_{H^1(B_{R_0})\times H^4(-\pi,\pi)}\leq C\|U_0\|_{H^2(B_{R_0})\times H^4(-\pi,\pi)}\quad \forall U_0\in
H^2_N (B_{R_0})\times H^4(-\pi,\pi),
$$
$$
\|U(t)\|_{H^1(B_{R_0})\times H^4(-\pi,\pi)}\leq  C T \int_0^T \| F(t)\|^2_{ L^2(B_{R_0})\times 
H^{5/2}(-\pi,\pi)}dt,\quad \forall F\in L^2(0,T;  L^2(B_{R_0})\times 
H^{5/2}(-\pi,\pi)), 
$$
where $U(t)=\int_0^te^{A(t-\tau)}F(\tau)d\tau$, $C$ is independent of $T$.
\end{cor}
\begin{proof} To get the sought bound for the  $\rho$-component, we write
$$
\begin{aligned}
\|\rho(t)\|^2_{H^4(-\pi,\pi)}=&C\|\rho(0)\|^2_{H^4(-\pi,\pi)}+
C\int_0^t \int_{-\pi}^{\pi}(\partial^4_{\varphi^4}\rho \partial_t \partial^4_{\varphi^4}\rho + \rho \partial_t \rho)d\varphi dt\\
& \leq 
C\|\rho(0)\|^2_{H^4(-\pi,\pi)}+C\int_0^t \|\partial_t \rho\|_{H^{5/2}(-\pi,\pi)} 
\|\rho\|_{H^{11/2}(-\pi,\pi)}dt
\end{aligned}
$$
and then use bounds from Lemma \ref{Dlinnaya_i_pechalnaya_lemma}. The $m$-component is treated similarly.
\end{proof}

Although  the function $\phi$ appearing in the linearized problem  \eqref{RadSymLinearized1}--\eqref{ur_e_dlya_m} does not belong to the 
phase space, it is convenient to introduce the operator $S_\phi(m,\rho)$ which 
assigns to the given $m$ and $\rho$ the unique solution $S_\phi(m,\rho)$of the problem 
\begin{equation}
\label{funktsiya_phi_iz_problemy} 
\Delta S_\phi + m=\zeta S_\phi+p_{\rm eff}^\prime(\pi R^2)R\int_{-\pi}^{\pi}\rho(\varphi)d\varphi \quad\text{in}\ B_R,\quad 
S_\phi=\frac{\gamma}{R^2\zeta}(\rho^{\prime\prime}+\rho) \quad \text{on}\ \partial B_R.
\end{equation}

To deal with shift invariance we rewrite problem  \eqref{actflow_in_terms_of_phi}--\eqref{myosin_bc_I} in moving frame 
with center at $\ve(X_{c,\ve}(t), 0)$, then $\Omega_\ve(t)=\tilde\Omega_\ve(t)+\ve(X_{c,\ve}(t), 0)$ and \eqref{actin_bc_normal}  after introducing the polar coordinates $(\tilde r,\tilde \varphi)$ to parameterize $\tilde\Omega_\ve(t)$, 
 $\tilde\Omega_\ve(t)=\{0\leq \tilde r<R_0+\ve \rho_\ve(\tilde\varphi,t)\}$, reads  
\begin{equation}\label{actin_bc_normal_new}
\partial_t \rho_{\ve}+\dot X_{c,\ve}\left(\cos\varphi +\frac{\ve \rho^\prime_{\ve}\sin \varphi}{R_0+\ve\rho_{\ve}}\right)=
\frac{\sqrt{\ve^2(\rho^\prime_{\ve})^2+(R_0+\ve\rho_{\ve})^2}}{\ve(R_0+\ve\rho_{\ve})}\partial_\nu \phi  \quad\text{on}\ \partial \tilde\Omega_\ve(t),
\end{equation}
while \eqref{myosin_equations_I} becomes
\begin{equation}
\label{meqOdIn}
\partial_t m= \Delta m +\ve\dot X_{c,\ve}\partial_x m-\div (m \nabla \phi), \quad \text{in} \ \tilde\Omega_\ve(t).
\end{equation}
We impose the orthogonality condition $\int_{-\pi}^{\pi}\partial_t \rho_\ve \cos\varphi d\varphi=0$ which yields the following equation governing the evolution of $X_{c,\ve}$
\begin{equation}
\label{evolution_of_center}
\dot X_{c,\ve}\left(1 +\frac{\ve}{2\pi}\int_{-\pi}^{\pi}\frac{\rho^\prime_{\ve}\sin 2\varphi}{R_0+\ve\rho_{\ve}}d\varphi\right)=
\frac{1}{\pi}\int_{-\pi}^{\pi}\frac{\sqrt{\ve^2(\rho^\prime_{\ve})^2+(R_0+\ve\rho_{\ve})^2}}{\ve(R_0+\ve\rho_{\ve})}\partial_\nu \phi \cos\varphi d\varphi.
\end{equation}

Next we introduce a transformation to reduce the study of the free boundary problem  to a problem in the fixed disk.  
We introduce local coordinates in an inner neighborhood of   $\partial \tilde\Omega_\ve$ by setting 
$
\tilde\Omega_\ve\ni (\tilde x,\tilde y)\mapsto (r,\varphi)\in (2R_0/3,R_0)\times (-\pi,\pi)$,
\begin{equation}
\label{change_in_neigh}
\begin{aligned}
\tilde x
=(R_0+\ve\rho_\ve(\varphi,t))\cos\varphi +(r-R_0)\frac{\ve \rho^\prime \sin\varphi+(R_0+\ve\rho_\ve)\cos\varphi}
{\sqrt{\ve^2(\rho^\prime_\ve)^2+(R_0+\ve\rho_\ve)^2}}, \\
\tilde y =(R_0+\ve\rho_\ve(\varphi,t))\sin\varphi +(r-R_0)\frac{-\ve\rho^\prime_\ve \cos\varphi+(R_0+\ve\rho_\ve)\sin\varphi}
{\sqrt{\ve^2(\rho^\prime_\ve)^2+(R_0+\ve\rho_\ve)^2}}      
\end{aligned}
\end{equation}
Note that the  normal vector on the boundary is given by 
$$\nu_x=\frac{\ve\rho^\prime_\ve \sin\varphi+(R_0+\ve\rho_\ve)\cos\varphi}{\sqrt{\ve^2(\rho^\prime_\ve)^2+(R_0+\ve\rho_\ve)^2}},
\quad
\nu_y=\frac{-\ve\rho^\prime_\ve \cos\varphi+(R_)+\ve\rho_\ve)\sin\varphi}{\sqrt{\ve^2(\rho^\prime_\ve)^2+(R_)+\ve\rho_\ve)^2}}.
$$
Also observe that  $R_0-r$ in \eqref{change_in_neigh}  
represents the distance from the boundary $\partial\tilde\Omega_\ve$ to $(\tilde x,\tilde y)$  and therefore the normal derivative on the boundary $\partial\tilde{\Omega}_\ve(t)$   becomes  the derivative  in $r$  on $\partial B_{R_0}$ even though the domain  $\tilde\Omega_\ve$  is obtained by non-radial perturbations  of the disk $B_R$:
$$
\partial_\nu m(\tilde x(R_0,\varphi),\tilde y(R_0,\varphi))=
\partial_r m(\tilde x(r,\varphi), \tilde y(r,\varphi))\Bigl|_{r=R_0}\Bigr.
$$
 In order to avoid  singular behavior at the origin, the coordinate transformation  \eqref{change_in_neigh}  is defined in a neighborhood of the boundary  $\partial \tilde\Omega_\ve$.  Then  the extension  from  this neighborhood to the entire domain   $\tilde \Omega_{\ve}$ by $(r,\varphi)\in [0,R_0)\times [-\pi,\pi)$ is done  by  employing  a  cutoff function  
$\chi\in C^\infty$,
$\chi(r)=1$ for $r>2R_0/3$ and $\chi(r)=0$ for $r<R_0/2$: 
\begin{equation*}
\begin{aligned}
\tilde x 
=\Bigl[(R_0+\ve\rho_\ve(\varphi,t))\cos\varphi +(r-R_0)\frac{\ve \rho^\prime_\ve \sin\varphi+(R_0+\ve\rho_\ve)\cos\varphi}
{\sqrt{\ve^2(\rho^\prime_\ve)^2+(R_0+\ve\rho_\ve)^2}}\Bigr]\chi(r)+(1-\chi(r))r\cos\varphi, \\
\tilde y =
\Bigl[(R_0+\ve\rho_\ve(\varphi,t))\sin\varphi +(r-R_0)\frac{-\ve\rho^\prime_\ve \cos\varphi+(R_0+\ve\rho_\ve)\sin\varphi}
{\sqrt{\ve^2(\rho^\prime_\ve)^2+(R_0+\ve\rho_\ve)^2}}
\Bigr]\chi(r)+
(1-\chi(r))r\sin\varphi,     
\end{aligned}
\end{equation*}
or
\begin{equation}
\label{ne_hanzawa}
\tilde x
=(r+\ve\eta)\cos\varphi-\ve \sigma \sin\varphi,\quad 
\tilde y =(r+\ve\eta)\sin\varphi +\ve\sigma\cos\varphi,
\end{equation}
where 
$$
\eta(r,\rho_\ve,\rho^\prime_\ve)=\frac{\ve(\rho_\ve^\prime)^2(R_0-r)}
{\Bigl(R_0+\ve\rho_\ve+{\sqrt{\ve^2(\rho^\prime_\ve)^2+(R_0+\ve\rho_\ve)^2}}\Bigr){\sqrt{\ve^2(\rho^\prime_\ve)^2+(R_0+\ve\rho_\ve)^2}}}\chi(r)
+\rho_\ve\chi(r),
$$
$$
\sigma(r,\rho_\ve,\rho^\prime_\ve)=(R_0-r)\frac{\rho^\prime_\ve}{\sqrt{\ve^2(\rho^\prime_\ve)^2+(R_0+\ve\rho_\ve)^2}}\chi(r).
$$


 Represent $m$ and $\phi$ in the  form 
$$
\begin{aligned}
m(\tilde x(r,\varphi,t)+\ve X_{c,\ve}(t),\tilde y(r,\varphi,t),  t)&=m_0+\ve m_\ve(r,\varphi,t), \\
\quad \phi(\tilde x(r,\varphi,t)+\ve X_{c,\ve}(t),\tilde y(r,\varphi,t),t)&=\phi_0+\ve \phi_\ve (r,\varphi,t)
+\ve p_{\rm eff}^\prime(\pi R^2_0)\frac{R_0}{\zeta}\int_{-\pi}^{\pi}\rho_\ve(\tilde\varphi,t)d\tilde\varphi,
\end{aligned}
$$
then \eqref{actflow_in_terms_of_phi}--\eqref{myosin_bc_I} rewrites as a problem whose  linear part is the same as in \eqref{RadSymLinearized1}--\eqref{ur_e_dlya_m}, but with additional nonlinear terms $f_1$, $f_2$, $g_1$, and $g_2$:   
\begin{equation}
\label{NEWactflow_in_terms_of_phi}
\Delta\phi_\ve +m_\ve=\zeta  \phi_\ve +p_{\rm eff}^\prime(\pi R^2_0)R_0\int_{-\pi}^{\pi}\rho_\ve(\tilde\varphi)d\tilde\varphi
 -\ve f_1 \quad\text{in}\  B_R,
\end{equation}
\begin{equation}\label{NEWactin_bc_potential}
\phi_\ve=\frac{\gamma}{R^2\zeta}(\rho^{\prime\prime}+\rho) 
+\ve g_1
\quad\text{on}\ \partial B_R,
\end{equation}
\begin{equation}\label{NEWactin_bc_normal}
\partial_t\rho_\ve =\partial_r  \phi_\ve-\frac{\cos\varphi}{\pi} \int_{-\pi}^\pi \partial_r \phi_\ve \cos\tilde\varphi d\tilde\varphi+\ve g_2\quad
\text{on}\  \partial B_R,
\end{equation}
\begin{equation}
\label{NEWmyosin_equations_I}
\partial_t m_\ve= \Delta m_\ve-m_0 \Delta\phi_\ve +\ve f_2
\quad \text{in} \ B_R,
\end{equation}
\begin{equation}
\label{NEWmyosin_bc_I}
\partial_r m_\ve =0 \quad 
%
\text{on}\ \partial B_R,
\end{equation}
together with
\begin{equation}
\label{skorost'}
\dot X_{c,\ve}=\tilde V,\quad \tilde V[\phi_\ve, \rho_\ve]= \frac{1}{
\pi \left(1 +\frac{\ve}{2\pi}\int_{-\pi}^{\pi}\frac{\rho^\prime_{\ve}\sin 2\varphi}{R_0+\ve\rho_{\ve}}d\varphi\right)}
\int_{-\pi}^{\pi}
\frac{\sqrt{\ve^2(\rho^\prime_\ve)^2+(R_0+\ve\rho_\ve)^2}}{R_0+\ve\rho_{\ve}}\partial_r\phi_\ve \cos\varphi d\varphi.
\end{equation} 
The additional term $f_1$  in \eqref{NEWactflow_in_terms_of_phi} appears  when applying the  coordinate change \eqref{ne_hanzawa} to  \eqref{actflow_in_terms_of_phi} and linearizing $p_{\rm eff}(|\tilde\Omega_\ve|)$,
\begin{equation*}
f_1[\tilde\phi_\ve,\rho_\ve]
=\frac{1}{\ve}p_{\rm eff}^\prime(\pi R^2_0)R_0\int_{-\pi}^{\pi}\rho_\ve(\tilde\varphi,t)d\tilde\varphi
-\frac{1}{\ve^2}
\left(
p_{\rm eff}\Bigl[
\int_{-\pi}^\pi (R_0+\ve \rho_\ve(\tilde\varphi,t))^2\frac{d\tilde\varphi}{2}
\Bigr] -p_{\rm eff}(\pi R_0^2)
\right)+
L(\phi_\ve,\rho_\ve),
\end{equation*}
where
\begin{equation*}
\begin{aligned}
 L(\phi_\ve, \rho_\ve)=&2(a_x\cos\varphi+a_y\cos\varphi)\partial^2_{rr}\phi_\ve
+\ve (a_x^2+a_y^2)
\partial^2_{rr} \phi_\ve\\
&+(\cos\varphi \partial_r a_x- b_x \sin\varphi+\frac{1}{r}(\cos\varphi \partial_\varphi a_y-\sin\varphi\partial_\varphi a_x) 
+\sin\varphi \partial_r a_y+ b_y \cos\varphi)\partial_r\phi_\ve\\
&+\ve(a_x\partial_r a_x+b_x\partial_\varphi a_x+a_y\partial_ra_y+b_y\partial_\varphi a_y)\partial_r\phi_\ve\\
&+2(b_x\cos\varphi-a_x\frac{\sin\varphi}{r}+a_y\frac{\cos\varphi}{r}+b_y\sin\varphi) \partial_{r\varphi}\phi_\ve+2\ve(a_xb_x+a_yb_y)\partial_{r\varphi}\phi_\ve\\
&+(a_x\frac{\sin\varphi}{r^2}+\cos\varphi \partial_r b_x-b_x\frac{\cos\varphi}{r}-\frac{\sin\varphi}{r}\partial_\varphi b_x
    -a_y\frac{\cos\varphi}{r^2}+\sin\varphi \partial_r b_y-b_y\frac{\sin\varphi}{r}+\frac{\cos\varphi}{r}\partial_\varphi b_y)
\partial_{\varphi}\phi_\ve\\
&+\ve(a_x\partial_r b_x+b_x\partial_\varphi b_x+a_y\partial_r b_y+b_y\partial_\varphi b_y)\partial_{\varphi}\phi_\ve\\
&+2(b_y\frac{\cos\varphi}{r}-b_x\frac{\sin\varphi}{r})\partial^2_{\varphi\varphi}m_\ve+\ve (b_x^2+b_y^2)\partial^2_{\varphi\varphi}
\phi_\ve,
\end{aligned}
\end{equation*}
with
\begin{equation*}
\begin{aligned}
a_x&=(\cos\varphi(\eta+\partial_\varphi \sigma)+\sin\varphi(\partial_\varphi \eta-\sigma))(\frac{1}{r}-\ve Z)- rZ\cos\varphi,\\
a_y&=(\cos\varphi(\sigma-\partial_\varphi \eta)+\sin\varphi(\eta+\partial_\varphi\sigma))(\frac{1}{r}-\ve Z)- rZ\sin\varphi,\\
b_x&=(\ve Z-\frac{1}{r})(\sin\varphi\partial_r\eta+\cos\varphi\partial_r\sigma)+Z\sin\varphi,\  
b_y=(\frac{1}{r}-\ve Z)(\cos\varphi\partial_r\eta-\sin\varphi\partial_r\sigma)-Z\sin\varphi,
\end{aligned}
\end{equation*}
$$
Z=\frac{\eta+\partial_r \eta +(1+\ve \partial_r \eta)\partial_\varphi \sigma +\ve \partial_r \sigma (\sigma-\partial_\varphi\eta)}
{r(r+\ve(\eta+\partial_r \eta +(1+\ve \partial_r \eta)\partial_\varphi \sigma) +\ve^2 \partial_r \sigma (\sigma-\partial_\varphi\eta))}.
$$

The term $f_2$ in \eqref{NEWmyosin_equations_I} appears when applying change
of variables \eqref{ne_hanzawa} to \eqref{myosin_equations_I},
\begin{equation*}
\begin{aligned}
f_2[m_\ve,\phi_\ve,\rho_\ve]=&L(m_\ve,\rho_\ve)-(m_0+\ve m_\ve)L(\phi_\ve, \rho_\ve)-m_\ve\Delta \phi_\ve
\\
&+
(\tilde V[\phi_\ve,\rho_\ve]+\cos\varphi\partial_t \eta-\sin\varphi\partial_t \sigma)
((\cos\varphi+\ve a_x)\partial_r m_\ve +
(\ve b_x-\frac{\sin\varphi}{r})\partial_\varphi m_\ve)
\\
&+(\sin\varphi\partial_t \eta+\cos\varphi\partial_t \sigma)
((\sin\varphi+\ve a_y)\partial_r m_\ve +
(\ve b_y+\frac{\cos\varphi}{r})\partial_\varphi m_\ve)
\\
&
-\left ((\cos\varphi+\ve a_x)\partial_r m_\ve +
(\ve b_x-\frac{\sin\varphi}{r})\partial_\varphi m_\ve\right)
\left((\cos\varphi+\ve a_x)\partial_r \tilde\phi_\ve +
(\ve b_x-\frac{\sin\varphi}{r})\partial_\varphi \phi_\ve\right)
\\
&
-\left ((\sin\varphi+\ve a_y)\partial_r m_\ve +
(\ve b_y+\frac{\cos\varphi}{r})\partial_\varphi m_\ve\right)
\left((\sin\varphi+\ve a_y)\partial_r \tilde\phi_\ve +
(\ve b_y+\frac{\cos\varphi}{r})\partial_\varphi \phi_\ve\right).
\end{aligned}
\end{equation*}
Also, 
\begin{equation}
\label{funktsiya_g1}
\begin{aligned}
g_1[\rho_\ve]=-\frac{2\gamma(\rho_\ve^\prime)^2}{\zeta{((R_0+\ve\rho_\ve)^2+\ve^2(\rho^\prime_\ve)^2)^{3/2}}}-
\frac{\gamma((\rho_\ve^{\prime\prime}+\rho_\ve)(2\rho_\ve R_0+\ve\rho_\ve^2)-R_0\rho_\ve^2)(R_0+\ve\rho_\ve)}
{\zeta R_0^2{((R_0+\ve\rho_\ve)^2+\ve^2(\rho^\prime_\ve)^2)^{3/2}}}\\
+\frac{\gamma(R_0-\ve\rho_\ve^{\prime\prime}-\ve\rho_\ve)(\rho^\prime_\ve)^2(2(R_0+\ve\rho_\ve)^2+(\ve\rho_\ve^\prime)^2
+
(R_0+\ve\rho_\ve)\sqrt{(R_0+\ve\rho_\ve)^2+(\ve\rho^\prime_\ve)^2})}
{\zeta R_0^2(R_0+\ve\rho_\ve+\sqrt{(R_0+\ve\rho_\ve)^2+(\ve\rho^\prime_\ve)^2}){((R_0+\ve\rho_\ve)^2+\ve^2(\rho^\prime_\ve)^2)^{3/2}}}
\end{aligned}
\end{equation}
\begin{equation}
\label{funktsiya_g2}
\begin{aligned}
g_2[\phi_\ve,\rho_\ve]=&\frac{\ve(\rho_\ve^\prime)^2}{\Bigl(R_0+\ve\rho_\ve+{\sqrt{\ve^2(\rho^\prime_\ve)^2+(R_0+\ve\rho_\ve)^2}}\Bigr)
{(R_0+\ve\rho_\ve)}}\partial_r\phi_\ve
\\
&+
 \frac{ \int_{-\pi}^\pi \partial_r \phi_\ve \cos\tilde\varphi d\tilde\varphi}
{
\pi \left(1 +\frac{\ve}{2\pi}\int_{-\pi}^{\pi}\frac{\rho^\prime_{\ve}\sin 2\tilde\varphi}{R_0+\ve\rho_{\ve}}d\tilde\varphi\right)}
\left[\frac{\cos\varphi}{2\pi}\int_{-\pi}^{\pi}\frac{\rho^\prime_{\ve}\sin 2\tilde\varphi}{R_0+\ve\rho_{\ve}}d\tilde\varphi
-\frac{\rho^\prime_{\ve}\sin \varphi}{R_0+\ve\rho_{\ve}}\right]
\\
&-\ve 
 \frac{\cos\varphi +\frac{\rho^\prime_{\ve}\sin \varphi}{R_0+\ve\rho_{\ve}}}
{
\pi \left(1 +\frac{\ve}{2\pi}\int_{-\pi}^{\pi}\frac{\rho^\prime_{\ve}\sin 2\tilde\varphi}{R_0+\ve\rho_{\ve}}d\tilde\varphi\right)}
\int_{-\pi}^\pi\frac{(\rho_\ve^\prime)^2\partial_r\phi_\ve \cos\tilde\varphi d\tilde\varphi}{\Bigl(R_0+\ve\rho_\ve+{\sqrt{\ve^2(\rho^\prime_\ve)^2+(R_0+\ve\rho_\ve)^2}}\Bigr)
{(R_0+\ve\rho_\ve)}}.
\end{aligned}
\end{equation}
%
The nonlinear terms $f_1$, $f_2$, $g_1$, $g_2$ in system \eqref{NEWactflow_in_terms_of_phi}--\eqref{skorost'} contain higher order derivatives, that is why regularity result (iii) in Lemma \ref{Dlinnaya_i_pechalnaya_lemma} is crucial for the solvability
of this system.

The solvability of  \eqref{NEWactflow_in_terms_of_phi}--\eqref{skorost'} is shown iteratively via the contraction mapping theorem. Namely 
in the initial step we solve \eqref{NEWactflow_in_terms_of_phi}--\eqref{NEWmyosin_bc_I} with given initial data and 
$f_1=f_2=0$, $g_1=g_2=0$ to obtain the first iteration 
$(m_{\ve,0},\rho_{\ve,0})=e^{\mathcal{A}t}(m_\ve(r,\varphi,0),
\rho_\ve(\varphi,0))$, $\phi_{\ve,0}=S_{\phi}(m_{\ve,0},\rho_{\ve,0})$. Without loss of generality we can assume that 
the function $\rho_\ve(\varphi,0)$, which determines the initial shape, is orthogonal to $\cos\varphi$, for otherwise one modifies appropriately 
the initial  position of the center $X_{c,\ve}(0)$ of the reference steady state. Then semigroup $e^{\mathcal{A}t}$ is well defined for such initial data. 

Next introduce 
new unknowns 
$\mu_\ve:=m_\ve-m_{\ve,0}$,  $\varrho_\ve:=\rho_\ve -\rho_{\ve,0}$ 
 and represent $\phi_\ve$ as $\phi_\ve=\psi_\ve+\phi_{\ve,0}+\ve\overline\psi_\ve$, where $\overline\psi_\ve=
\overline\psi_\ve[\psi_{\ve},\varrho_{\ve},\psi_{\ve,0},\rho_{\ve,0}]$ solves
\begin{equation}
\label{vspomog_actin}
\Delta\overline {\psi}_\ve=\zeta \overline{\psi}_\ve 
 -f_1[\psi_\ve+\phi_{\ve,0}+\ve\overline\psi_\ve,\varrho_\ve+\rho_{\ve,0}]\ \quad\text{in}\  B_R,
\end{equation}
\begin{equation}\label{vspomogbc1}
\overline{\psi}_\ve=  g_1[\varrho_\ve+\rho_{\ve,0}]
\quad\text{on}\ \partial B_R,
\end{equation}
to rewrite 
equations \eqref{NEWactflow_in_terms_of_phi}--\eqref{NEWmyosin_bc_I} in the form
\begin{equation}
\label{strasnenkii_actin}
\Delta\psi_\ve +\mu_\ve=\zeta \psi_\ve +p_{\rm eff}^\prime(\pi R^2_0)R_0\int_{-\pi}^{\pi}\varrho_\ve(\tilde\varphi)d\tilde\varphi 
\quad\text{in}\  B_R,
\end{equation}
\begin{equation}\label{strasnenkiibc1}
\psi_\ve=\frac{\gamma}{R^2\zeta}(\varrho^{\prime\prime}_\ve+\varrho_\ve) 
\quad\text{on}\ \partial B_R,
\end{equation}
\begin{equation}\label{strasnenkii_bc_normal}
\begin{aligned}
\partial_t\varrho_\ve =\partial_r \psi_\ve-\frac{\cos\varphi}{\pi} \int_{-\pi}^\pi \partial_r \psi_\ve \cos\tilde\varphi d\tilde\varphi+\ve
\Bigl(\partial_r \overline\psi_\ve&-\frac{\cos\varphi}{\pi} \int_{-\pi}^\pi \partial_r \overline\psi_\ve \cos\tilde\varphi d\tilde\varphi 
\Bigr.\\
&+\Bigl. g_2[\psi_\ve
+\phi_{\ve,0}+\ve\overline\psi_\ve,\varrho_\ve+\ve\rho_{0,\ve}]\Bigr)
\end{aligned}
 \quad
\text{on}\  \partial B_R,
\end{equation}
\begin{equation}
\label{strasnenkii_myosin}
\partial_t \mu_\ve= \Delta \mu_\ve-m_0 \Delta \psi_\ve +
\ve 
(
f_2[\mu_\ve+m_{\ve,0},\psi_\ve+\phi_{\ve,0}+\ve\overline\psi_\ve,\varrho_\ve+\rho_{\ve,0}]
- 
m_0\Delta\overline\psi_\ve
)
\quad \text{in} \ B_R,
\end{equation}
\begin{equation}
\label{strasnenkii_myobc}
\partial_r \mu_\ve =0 \quad 
%
\text{on}\ \partial B_R.
\end{equation}
Thus by Duhamel's formula we have 
\begin{equation}
\label{Banach_fix_point}
(\mu_\ve,\varrho_\ve)=\ve \int_0^t e^{\mathcal{A}(t-\tau)}
(\tilde f_2(\mu_\ve,\varrho_\ve), \tilde g_2 ( \mu_\ve, \varrho_\ve))d\tau=:G_\ve(\mu_\ve,\varrho_\ve),
\end{equation}
where 
$$
\tilde g_2 =  g_2[\psi_\ve
+\phi_{\ve,0}+\ve\overline\psi_\ve,\varrho_\ve+\rho_{\ve,0}]+ 
\partial_r \overline \psi_\ve -\frac{\cos\varphi}{\pi} \int_{-\pi}^\pi \partial_r \overline\psi_\ve \cos\tilde\varphi d\tilde\varphi,
$$
$$
\tilde f_2= f_2[\mu_\ve+m_{\ve,0},\psi_\ve+\phi_{\ve,0}+\ve\overline\psi_\ve,\varrho_\ve+\rho_{\ve,0}]
- 
m_0\Delta\overline\psi_\ve,
$$
and $\psi_\ve=S_{\phi}(\mu_\ve, \varrho_\ve)$ in the definition \eqref{vspomog_actin}--\eqref{vspomogbc1} of  $\overline{\psi}_\ve$. 
The fixed point problem \eqref{Banach_fix_point} is considered in the space 
\begin{equation}
\label{prostir}
\begin{aligned}
Y=\Bigl\{&
(\mu,\varrho)\in L^\infty([0,T];\, H^1(B_{R_0}))\times L^\infty([0,T];\, H^4_{\rm per}(-\pi,\pi)\setminus\{\cos\varphi \});\Bigr.
\\
& (\mu,\varrho)\in  L^2([0,T];\, H^2_{N}(B_{R_0}))\times L^2([0,T];\, H^{11/2}_{\rm per}(-\pi,\pi)))
\\
&\Bigl.\partial_t (\mu,\varrho)\in  L^2([0,T];\, L^2(B_{R_0}))\times L^2([0,T];\, H^{5/2}_{\rm per}(-\pi,\pi))),\,  (\mu,\varrho)|_{t=0}=0\Bigr\},
\end{aligned}
\end{equation}
endowed with the norm
$
\|(\mu,\varrho)\|_Y^2=\|(\mu,\varrho)\|_{1}^2+\|(\mu,\varrho)\|_{2}^2+\|\partial_t(\mu,\varrho)\|_{3}^2
$, where $\|\,\cdot\,\|_{1}$,$\|\,\cdot\,\|_{2}$ and $\|\,\cdot\,\|_{3}$ denote norms in 
$ L^\infty([0,T];\, H^1(B_{R_0}))\times L^\infty([0,T];\, H^4(-\pi,\pi)\})$,
$L^2([0,T];\, H^2_{N}(B_{R_0}))\times L^2([0,T];\, H^{11/2}(-\pi,\pi)))$ and
$ L^2([0,T];\, L^2(B_{R_0}))\times L^2([0,T];\, H^{5/2}(-\pi,\pi)))$, respectively, $ H^2_{N}(B_{R_0})$
 is the subspace of the Sobolev space $H^2(B_{R_0})$ of functions $\mu$ satisfying $\partial_r \mu=0$ on  $\partial B_{R_0}$. 

By Lemma \ref{Dlinnaya_i_pechalnaya_lemma} and Corollary \ref{cor_pro_l_infty_bounds} we have the following bound 
\begin{equation*}
 \|(m_{\ve,0},\rho_{\ve,0})\|_{Y}\leq C_0 I_{0,\ve},
\end{equation*}
where 
$$
I_{0,\ve}:=\|(m_{\ve,0},\rho_{\ve,0})\|_{H^2(B_{R_0})\times H^4(-\pi,\pi)}\bigl|_{t=0}\leq 1,
$$
and $C_0$ is independent of $T$. Consider the set 
\begin{equation*}
 E=\{(\mu,\varrho)\in Y;\, \|(\mu,\varrho)\|_{Y}\leq 2C_0 I_{0,\ve}\}.
\end{equation*}
We next show that $G_\ve$ defined in \eqref{Banach_fix_point} maps the set $E$ into itself for sufficiently small $\ve$, moreover
\begin{equation}
\label{okonchat_pobeda}
\|G_\ve(\mu,\varrho)\|_{Y}\leq \ve C(1+T) I_{0,\ve}, \quad \forall (\mu,\varrho)\in E,
\end{equation}
where $C$ is independent of $T$ and $I_{0,\ve}$. 
 To this end 
observe that  the mappings $g_1[\rho_\ve]$, $g_2[\phi_\ve, \rho_\ve]$ and $f_1[\rho_\ve,\phi_\ve]$
have the following pointwise in $t\in[0,T]$ bounds 
$$
\|g_1[\rho_\ve]\|_{H^{k-2}(-\pi,\pi)}\leq C\|\rho_\ve\|_{H^{k}(-\pi,\pi)},\quad k=4,\, 11/2
$$
$$
\|g_2[\phi_\ve, \rho_\ve]\|_{H^{5/2}(-\pi,\pi)}\leq C\|\phi_\ve\|_{H^4(B_{R_0})},
$$
$$
\|f_1[\phi_\ve,\rho_\ve]\|_{H^{k-2}(B_{R_0})}\leq C\left(\|\phi_\ve\|_{H^{k}(B_{R_0})}+\|\rho_\ve\|_{H^{4}(-\pi,\pi)}\right),
\quad k=2,\, 5/2 
$$
$$
\begin{aligned}
\|f_2[m_\ve,\phi_\ve,\rho_\ve]\|_{L^2(B_{R_0})}\leq& C \Bigl(\|m_\ve\|_{H^2(B_{R_0})}+\|\phi_\ve\|_{H^{5/2}(B_{R_0})}\Bigr)
\end{aligned}
$$
and the integral bound
$$
\begin{aligned}
\|f_1[\phi_\ve,\rho_\ve]\|^2_{L^2(0,T; H^2(B_{R_0}))}\leq C &\Bigl(\|\phi_\ve\|^2_{L^2(0,T; H^4(B_{R_0}))}\Bigr.
\\
&+\Bigl.\|\rho_\ve\|_{L^2(0,T;H^{5}(-\pi,\pi))}^2\left(1+\|\phi_\ve\|^2_{L^\infty(0,T; H^{5/2}(B_{R_0}))}\right)\Bigr)
\end{aligned}
$$
for $(m_\ve,\rho_\ve)\in E$ and $\ve\leq \ve_0$ (with some $\ve_0>0$). Then \eqref{okonchat_pobeda} follows by 
Lemma \ref{Dlinnaya_i_pechalnaya_lemma} and Corollary \ref{cor_pro_l_infty_bounds}. 
Also, one checks that $G_\ve$ is Lipschitz continuous with Lipschitz constant less than one for sufficiently small $\ve$. Thus there 
exists the unique fixed point $(\mu_\ve,\varrho_\ve)$ and we have 
$$
(m_\ve,\rho_\ve)=C_{0,\ve} U_1/\Pi + \tilde U_\ve+(\mu_\ve,\varrho_\ve),
$$
where $|C_{0,\ve}|\leq CI_{0,\ve}$ and
 $\|\tilde U_\ve\|_{H^2(B_{R_0})\times H^4(-\pi,\pi)}\leq C_1 e^{-\theta t} I_{0,\ve}$. Here $\theta>0$ is constant appearing in \eqref{exp_decay_bis}. Thus there is 
$\frac{1}{\theta}\log\frac{1}{\ve} \leq T^\ast\leq \frac{1}{\theta}\log\frac{1}{\ve}+1 $ such that at $t=T^\ast$
$$
\|\tilde U_\ve+(\mu_\ve,\varrho_\ve)\|_{H^2(B_{R_0})\times H^4(-\pi,\pi)}\leq C \ve I_{0,\ve}\log \frac{1}{\ve}  
$$
On the other hand,  due 
to myosin preservation property we have 
\begin{eqnarray}
\nonumber&&\int_{\Omega_\ve(0)} m(x,y,0)dxdy=\int_{\Omega_\ve(t)} m(x,y,t)dxdy \\
\nonumber&&\hspace{50pt}=\int_{0}^{R_0}\int_{-\pi}^\pi \left(m_0+\ve m_\ve(r,\varphi,t)\right)J_\ve(r,\varphi)d\varphi dr
\\   
\nonumber&&\hspace{50pt}=|\Omega_\ve(t)|\left(m_0+\ve (\frac{\gamma}{R_0}+2\pi R_0 p'_h(\pi R_0^2))\frac{C_{0,\ve}}{\Pi}\right)+\ve\int_{0}^{R_0}\int_{-\pi}^\pi(\tilde{m}_{\ve,0}+\mu_\ve) J_\ve d\varphi dr
\\
\nonumber&&\hspace{50pt}=m_0\left(\pi R_0^2+2\pi R_0\ve \frac{C_{0,\ve}}{\Pi}\right)+\ve\left(\frac{\gamma}{R_0}+2\pi R_0p_{\rm eff}'(\pi R_0^2)\right)\frac{C_{0,\ve}}{\Pi}\pi R_0^2+O(\ve^2 I_{0,\ve}\log \frac{1}{\ve})
\\
\nonumber&&\hspace{50pt}= m_0(\pi R_0^2 +  \ve C_{0,\ve})
+O(\ve^2 I_{0,\ve}\log \frac{1}{\ve}).
%
\end{eqnarray} 
In equalities above, $\tilde{m}_{\ve,0}$ is the $m$-component of $\tilde{U}_{\ve}$ and we used the following expression for Jacobian: 
\begin{equation*}
J_\ve=
(1+\ve \partial_r\eta)(r+\ve \eta)+\ve \partial_{\varphi}\sigma(1+\ve \partial_r\eta)+\ve^2 \sigma \partial_{r}\sigma-\ve^2 \partial_r \sigma \partial_\varphi\eta=r+\ve \rho_{\ve,0}\partial_r (r\chi(r))+O\left(\ve^2I_{0,\ve}\log \frac{1}{\ve}\right). 
\end{equation*}

Since $\int_{\Omega_\ve(0)} m(x,y,0)dxdy=\pi m_0 R_0^2$, we get $C_{0,\ve}=O(\ve I_{0,\ve}\log \frac{1}{\ve})$. Thus, for sufficiently small $\ve$,  $\|(m_\ve,\rho_\ve)\|_{H^2(B_{R_0})\times H^4(-\pi,\pi)}\bigl|_{t=T^\ast}<\sqrt{\ve}\|(m_\ve,\rho_\ve)\|_{H^2(B_{R_0})\times H^4(-\pi,\pi)}|_{t=0}$. Applying this result iteratively we establish exponential decay of the solution as 
$t\to \infty$.
\end{proof}

\section*{Appendix}
\label{sec:appedix}

\begin{prop}\label{lem:ineq_for_m}
Consider $m\in H^1(B_R)$ such that it satisfies Neumann boundary condition \eqref{eig5} and $\langle m \cos \varphi \rangle=0$, where 
$\langle v\rangle:=\frac{1}{\pi R^2}\int_{B_R}v\,dxdy$. Then 
\begin{equation}\label{ineq_for_m}
\int_{B_R}|\nabla m|^2\, dxdy - m_0\int_{B_R}|m|^2 \,dxdy \geq -m_0\pi R^2|\langle m\rangle|^2
\end{equation}
for any $m_0$ which is less or equal to the third eigenvalue of the operator $-\Delta$ in $B_R$ with the Neumann boundary condition on 
$\partial B_R$.
\end{prop}
\begin{proof}
Similar to operator $\mathcal{A}$, eigenvectors for operator $-\Delta$ with Neumann boundary condition are of the form $m=\hat{m}(r)\cos(n\varphi)$ for integer $n\geq 0$, and for each $n\geq 0$ there are infinitely many eigenvalues. The first (minimal) eigenvalue of $-\Delta$ with Neumann boundary condition is $\lambda_1^{(N)}=0$ and the corresponding eigenvector $m$ is a constant ($n=0$ and $\hat{m}(r)\equiv \text{const}$). Let us show that the second eigenvalue $\lambda_2^{(N)}$ corresponds to an eigenvector $m$ of the
 form $m=\hat{m}(r)\cos (\varphi)$ ($n=1$).

First, we note that 
\begin{equation}
\lambda_2^{(N)}=\inf_{\footnotesize\begin{array}{c}m\in H^1(B_R)\\ \langle m\rangle =0\end{array}}
\dfrac{\int_{B_R}|\nabla m|^2\,dxdy}{\int_{B_R}|m|^2\,dxdy}=\inf_{\footnotesize\begin{array}{c}m=\hat{m}(r)\cos (n\varphi)\\ \langle m\rangle =0\end{array}}
\dfrac{\int\limits_0^R\left(|\hat{m}'|^2+\frac{n^2}{r^2}|\hat{m}|^2\right)r\,dr}{\int_{0}^{R}|\tilde{m}|^2r\,dr},\label{eig_34}
\end{equation}  
 where the second equality holds for some integer $n\geq 0$ and we aim to show that $n=1$ in \eqref{eig_34}. Indeed, since $\langle m\rangle=0$ for all $m$ of the form $m=\hat{m}(r)\cos (n\varphi)$ with $n\geq 1$, the minimum of fraction in the right hand side of \eqref{eig_34} among $n\geq 1$ is attained at $n=1$. Thus, $n=0$ or $n=1$. Assume that $n=0$. Then  the corresponding eigenfunction $m$ is of the form $m=\hat{m}(r)$. By straightforward calculations one shows that $u:=\hat{m}'(r)\cos \varphi$ is an eigenfunction of operator $-\Delta$ with Dirichlet boundary conditions for eigenvalue $\lambda^{(N)}_2$ and thus
 \begin{equation*}
 \lambda_2^{(N)}= \dfrac{\int_{B_R}|\nabla u|^2\,dxdy}{\int_{B_R}|u|^2\,dxdy}.
 \end{equation*}  
 Denoting by $\lambda^{*}$ minimal eigenvalue for operator $-\Delta$ with Neumann boundary condition corresponding to $n=1$, we obtain 
 \begin{equation*}
 \lambda^*=\inf_{\footnotesize m=\hat{m}(r)\cos (\varphi)}\dfrac{\int_{B_R}|\nabla m|^2\,dxdy}{\int_{B_R}|m|^2\,dxdy}<  
 \dfrac{\int_{B_R}|\nabla u|^2\,dxdy}{\int_{B_R}|u|^2\,dxdy}=\lambda^{(N)}_2,
 \end{equation*} 
where the strict inequality follows from the fact that $u=0$ on $\partial B_R$, a contradiction.     
 It follows, in particular, that for all $m\in H^1(B_R)$ such that $\langle m \rangle = \langle m \cos \varphi\rangle=0$ we have 
 \begin{equation}\label{ineq_for_m_ave}
 \int_{B_R}|\nabla m|^2 \, \text{d}x\text{d}y \geq \lambda^{(N)}_{3} \int_{B_R} |m|^2 \,dxdy\geq  m_0 \int_{B_R} |m|^2 \,dxdy.
 \end{equation} 
Now take an arbitrary $m\in H^1(B_R)$ such that $\langle m \cos \varphi\rangle=0$
and apply \eqref{ineq_for_m_ave} for $m-\langle m\rangle$:
\begin{equation*}
\int_{B_R}|\nabla m|^2 \, \text{d}x\text{d}y \geq  m_0 \int_{B_R} |m-\langle m\rangle|^2 \,dxdy=m_0 \int_{B_R} |m|^2\,dxdy-m_0\pi R^2\langle m\rangle^2.
\end{equation*}
Thus, \eqref{ineq_for_m} is proved. 
\end{proof}

\end{document}